\documentclass[11pt,a4paper,oneside, reqno]{amsart}
\usepackage{mathrsfs}
\usepackage{amsfonts}
\usepackage{amssymb}
\usepackage{amsxtra}
\usepackage{dsfont}
\usepackage{amsthm}
\usepackage{geometry}
\usepackage{float}
\usepackage{amsmath, amssymb, amsthm}
\usepackage[utf8]{inputenc}
\usepackage[T1]{fontenc}
\usepackage{graphicx}
\usepackage[psdextra]{hyperref}
\usepackage{listings}
\usepackage{indentfirst}
\usepackage{theoremref}
\usepackage{thmtools}
\usepackage{physics}
\usepackage{mathtools}
\usepackage[polish,english]{babel}
\usepackage{color}
\usepackage{bbm}
\usepackage{url}
\usepackage{hyperref}
\usepackage[hyphenbreaks]{breakurl}

\newcommand{\pl}[1]{\foreignlanguage{polish}{#1}}

\usepackage{hyphenat}

\newcommand{\innprod}[2]{\left< #1, #2 \right>}
\newcommand{\R}{\mathbb{R}}
\newcommand{\C}{\mathbb{C}}
\newcommand{\RR}{\mathbb{R}}

\newcommand{\Z}{\mathbb{Z}}
\newcommand{\E}{\mathbb{E}}

\renewcommand{\P}{\mathbb{P}}

\DeclareMathOperator{\erfc}{erfc}
\DeclareMathOperator{\Rea}{Re}

\DeclareMathOperator{\ed}{ED}
\DeclareMathOperator{\supp}{supp}
\DeclareMathOperator{\dom}{Dom}
\DeclareMathOperator{\dist}{dist}

\newcommand{\ind}[1]{\mathbbm{1}_{#1}}

\newcommand{\ag}{\approx_g}
\newcommand{\leg}{\leqslant_g}
\newcommand{\lesg}{\lesssim_g}
\newcommand{\gtrsg}{\gtrsim_g}

\theoremstyle{theorem}
\newtheorem{theorem}{Theorem}[section]
\newtheorem{theoremx}{Theorem}
\newtheorem{lemma}[theorem]{Lemma}
\newtheorem{pro}[theorem]{Proposition}
\newtheorem{cor}[theorem]{Corollary}

\theoremstyle{remark}
\newtheorem*{remark}{Remark}

\theoremstyle{definition}
\newtheorem{df}{Definition}[section]

\definecolor{green}{rgb}{0,0.75,0}

\newcommand{\la}{\lambda}

\numberwithin{equation}{section}

\title[On $L^{p}$ estimates for Riesz transforms related to Schr\"odinger operators]{On $L^p$ estimates for positivity-preserving Riesz transforms related to Schr\"odinger operators}

\author{Maciej Kucharski}
\address{Maciej Kucharski\\
	Instytut Matematyczny\\
	Uniwersytet \pl{Wroc{\l}awski}\\
	Plac Grun\-waldzki 2\\
	50-384 \pl{Wroc{\l}aw}\\
	Poland}
\email{mkuchar@math.uni.wroc.pl}

\author{B{\l}a{\.z}ej Wr{\'o}bel}
\address{ B{\l}a{\.z}ej Wr{\'o}bel\\
	Instytut Matematyczny, Polska Akademia Nauk, \'Sniadeckich 8,	00–656 Warszawa \& Instytut Matematyczny, Uniwersytet Wroc\l awski, pl. Grunwaldzki 2/4, 50-384 Wroc\l aw, Poland}
\email{blazej.wrobel@math.uni.wroc.pl}

\thanks{Both authors were supported by the National Science Centre (NCN), Poland  research project Preludium Bis 2019/35/O/ST1/00083.}

\begin{document}
	
	\begin{abstract}
		We study the  $L^{p},$ $1\leqslant p\leqslant \infty,$  boundedness for Riesz transforms of the form $V^{a}(-\frac12\Delta+V)^{-a},$ where $a>0$ and $V$ is a non-negative potential. 	We prove that $V^{a}(-\frac12\Delta+V)^{-a}$ is bounded on  $L^p(\R^d)$ with $1< p\leqslant 2$ whenever $a\leqslant 1/p.$   We  demonstrate that the $L^{\infty}(\R^d)$ boundedness holds if $V$ satisfies an $a$-dependent integral condition that is resistant to small perturbations. Similar results with stronger assumptions on $V$ are also obtained on $L^{1}(\R^d).$ In particular our $L^{\infty}$ and $L^1$ results apply to non-negative locally bounded potentials $V$ which globally have a power growth or an exponential growth.
		We also discuss a counterexample showing that the $L^{\infty}(\R^d)$ boundedness may fail.
	
	\end{abstract}
	
	\maketitle
	\section{introduction}
	\label{sec:int}

In this paper we consider a class of Riesz transforms related to the Schr\"{o}dinger operator
\[
L = -\frac{1}{2}\Delta + V,
\]
with $V$ being a non-negative potential in $L^1_{\rm loc}$. The operator $L$ is rigorously defined via quadratic forms, see Section \ref{sec:def}.  The Riesz transforms are formally given, for $a>0,$  by
\begin{equation}
\label{eq:Rdef}
R^a _Vf(x)= V^{a}(x)\cdot \left( -\tfrac{1}{2}\Delta+V\right)^{-a}f(x)=\frac{V^{a}(x)}{\Gamma(a)} \cdot \int_0^\infty e^{-tL}f(x) \ t^{a-1} dt,
\end{equation}
where $e^{-tL}$ is the corresponding semigroup. 
We also set $R^0_V$ to be the identity operator.
By the Trotter product formula the operators $R^{a}_V$ are positivity preserving, unlike the Riesz transforms $\nabla L^{-1/2}$, which we do not study here. One can also see, cf.\ \thref{pro:L2}, that for $V\in L^1_{\rm loc}$ and $a=\frac{1}{2}$ the formal expression \eqref{eq:Rdef} gives rise to a contraction on $L^2(\RR^d).$ For a general non-negative potential $V\in L^1_{\rm loc}$ we also know the $L^1(\R^d)$ boundedness of $R_V^1$, see for example \cite{AuBA, GallMor, katoLp}. Note that, apart from the case when $V$ is constant, neither $R^{1/2}_V$ nor $R^{1}_V$  is a convolution operator. 

Apart from the cases $a=\frac{1}{2}$ and $a=1$ there seem to be no $L^p$ boundedness results for Riesz transforms $R_V^a$ of general potentials $V\in L^1_{\rm loc}$. For $V$ belonging to a reverse H\"older class $L^p$ boundedness of $R_V^a,$ $0<a< 1$, is mentioned in  \cite[p.\ 1978]{AuBA}. We prove the following general result.
\begin{theoremx}[\thref{thm:genLp}]
	\thlabel{thm:genLpint}
	Let  $V\in L^1_{\rm loc}$ and take $p\in (1,2].$ Then for all $0\leqslant a \leqslant 1/p$ the Riesz transform $R_V^a$ is bounded on $L^p.$
\end{theoremx} 
\noindent \thref{thm:genLpint} is derived as a consequence of the endpoint bounds for $R_V^{1/2}$ on $L^2(\R^d)$ (\thref{pro:L2}) and for $R_V^1$ on $L^1(\R^d)$ (\cite[Theorem 4.3]{AuBA}, see also \cite{GallMor, katoLp}) together with the interpolation result given below.
\begin{theoremx}[\thref{thm:intL1}]
	\thlabel{thm:intLpint}
	Let $a_0>0$ and $a_1>0.$  Assume that $V\in L^1_{\rm loc}$ is such that $R_V^{a_1}$ is bounded on $L^{p_1}$ for some $p_1\in [1,\infty)$ and $R_V^{a_0}$ is bounded on $L^1.$ Then, $R_V^a$ is bounded on $L^p$ for every $p$ and $a$ such that $1/p=\theta+(1-\theta)/p_1$ and $a=\theta a_0 +(1-\theta)a_1$ with some $\theta\in (0,1).$
\end{theoremx}
\noindent The above theorem is proved via Stein's complex interpolation theorem. It is worth emphasizing that when $p\in (1,2]$ the boundedness of $R_V^a$ stated in \thref{thm:genLpint} holds not only for $a=1/p$ but for all smaller $a$ as well. This follows from \thref{thm:intLpint} together with \thref{cor:intLp}. However, this may be no longer true when $p=1.$ The reason behind is eminent in the proof of \thref{thm:intLpint} (\thref{thm:intL1}); namely, the imaginary powers $L^{iu},$ $u\in \RR$, are bounded on $L^p,$ $p\in (1,2]$, but are unbounded on $L^1.$

The main purpose of our paper is to study the $L^{\infty}$ and $L^{1}$  boundedness of $R_V^a$ for specific classes of non-negative potentials $V$. We focus on obtaining results for which only large values of $x$ matter and which are resistant to small perturbations of the potential $V.$ Considering the $L^{\infty}$ boundedness of $R^a_V$ two particular cases of $V$ serve as a good example of the possible situation. Firstly, if $V$ is a non-negative constant function, say $V=c$, then $L=-\tfrac{\Delta}{2}+c$ and by \eqref{eq:Rdef} we have 
\[
R^a_c f=\frac{c^a}{\Gamma(a)}\int_{0}^{\infty}e^{-tc}t^{a-1}e^{t\Delta/2}f\,dt.
\]
Therefore, using the $L^{\infty}$ contractivity of the heat semigroup $e^{t\Delta}$ we easily see that $R^a_c$ is bounded on $L^{\infty}.$ Secondly, if $d\geqslant 3$ and $V\in L^q,$ $q>d/2,$ is a non-zero compactly supported function, then $R^a_V$ is unbounded on $L^{\infty}$ for all $a>0,$ see \thref{pro:cexlinf}. Thus,  the fact that $V$ does not vanish outside a compact set is necessary for the boundedness of $R^a_V$ on $L^{\infty}.$

In what follows for two functions $A,B\colon \R^d\to [0,\infty)$ by $A(x)\approx B(x)$ we mean that for almost all $x \in \R^d$ we have $cA(x)\leqslant B(x)\leqslant CA(x)$ with two universal constants $0<c<C.$ We say that $A\approx B$ globally if $A(x)\approx B(x)$ for almost every $x$ outside a compact set.  The main classes of examples on $L^{\infty}(\R^d)$ which our theory admits are given below.
\begin{theoremx}
	\thlabel{thm:exLinfint}
	Let  $V\colon \R^d\to [0,\infty)$ be a function in $L_{\rm loc}^\infty$. Then in all the three cases
	\begin{enumerate}
		\item $V(x)\approx 1$ globally
		\item For some $\alpha>0$ we have $V(x)\approx |x|^{\alpha}$ globally
		\item For some $\beta>1$ we have $V(x)\approx \beta^{|x|}$ globally
	\end{enumerate}
	each of the Riesz transforms $R_V^a,$ $a>0$, is bounded on $L^{\infty}(\R^d).$ 
\end{theoremx} 
\noindent What lies at the heart of the proof of \thref{thm:exLinfint}  is the Feynman--Kac formula. \thref{thm:exLinfint} is restated as \thref{cor:linf} in Section \ref{sec:linf}, where it is deduced from  \thref{th:linf}. In order to apply \thref{th:linf} we need to verify two assumptions. Firstly $V$ must be strictly positive far away along a line in $\R^d.$ In this case \thref{lem21} guarantees an exponential decay of the semigroup $e^{-tL}$ on $L^{\infty}(\R^d)$. Secondly, we assume a specific interplay between the value $V(x)$ and the speed at which $V(y)$ decreases for $y$ in a ball around $x.$ The interplay is captured in condition \eqref{th:linf:eq1} (the quantity $I^a(V)(x)$ being defined in \eqref{eq:def:IaV}). It is easily verified that the assumptions of \thref{th:linf} are met in all the cases (1), (2), (3) of \thref{thm:exLinfint}.

We also prove an $L^1(\R^d)$ counterpart of \thref{thm:exLinfint} 
\begin{theoremx}
	\thlabel{thm:exL1int}
	Let  $V\colon \R^d\to [0,\infty)$ be a function in $L_{\rm loc}^\infty$. Then in all the three cases
	\begin{enumerate}
		\item $V(x)\approx 1$ globally
		\item For some $\alpha>0$ we have $V(x)\approx |x|^{\alpha}$ globally
		\item For some $\beta>1$ we have $V(x)\approx \beta^{|x|}$ globally
	\end{enumerate}
	each of the Riesz transforms $R_V^a,$ $a>0$, is bounded on $L^{1}(\R^d).$ 
\end{theoremx}
\noindent The proof of \thref{thm:exL1int} also makes extensive use of the Feynman--Kac formula. However, such an approach seems better suited to $L^{\infty}(\R^d)$ estimates and thus the route to \thref{thm:exL1int} is more complicated than in \thref{thm:exLinfint}. All the needed ingredients are justified in Section \ref{sec:l1}. \thref{thm:exL1int} is restated there as \thref{cor:l1} and the results needed to prove this corollary include \thref{th:l1}  and \thref{th:l1ho}. Note that in these results apart from condition \eqref{th:linf:eq1} we need to control the speed at which $V(y)$ increases for $y$ in a ball around $x$ relative to the value of $V(x)$. This is similar to the conditions assumed in the case of $L^\infty$ bounds.

Using \thref{thm:exLinfint,thm:exL1int} for $a=1,$ together with the argument from \cite[Proof of Corollary 1.4, pp.\ 174--175]{UrZien1}, we may also obtain $L^p(\R^d),$ $1<p<\infty$, boundedness of the Riesz transforms $|\nabla L^{-1/2} f|;$ here $\nabla$ denotes the usual gradient on $\R^d$. As this is aside the main considerations of our paper we do not pursue it here.


The topic of Riesz transforms related to Schr\"odinger operators has been considered by a number of authors, both on $\RR^d$ and on more general manifolds, see \cite{AO,AuBA,BadrBenAli,DeDiYao,Dev,Dz1,DzGl1,Sik1,UrZien1}.
In the context of the Riesz transforms $R_V^a$	the case $a=\frac{1}{2}$ has attracted most attention. For a general $V\in L^2_{\rm loc}$ it is known that $R^{1/2}_V$ is  bounded on the $L^p(\RR^d)$ spaces $1<p\leqslant 2,$ see Sikora \cite[Theorem 11]{Sik1}. Our \thref{thm:genLpint} extends the $L^p(\R^d)$ boundedness to $R^a_V$ for $a\leqslant 1/p.$ When the potential $V$ is in the reverse H\"older class $B_q$ for some $q\geqslant d/2,$ then Shen proved that $R_V^{1/2}$ is bounded on $L^p(\RR^d)$, $1 \leqslant p \leqslant 2q$, see \cite[Theorem 5.10]{Shen1}, and that $R_V^{1}$ is bounded on $L^p(\RR^d)$, $1 \leqslant p \leqslant q$, see \cite[Theorem 3.1]{Shen1}. Both results were later improved by Auscher and Ben Ali, see \cite[Theorem 1.1 and Theorem 1.2]{AuBA} to $1 < q \leqslant \infty$. In particular this is true for $V$ being a non-negative polynomial on $\RR^d.$ In fact, for such a $V$ the Riesz transforms $R^a_V,$ $a\geqslant 0,$ are bounded both on $L^1(\R^d)$ and $L^{\infty}(\R^d)$; this was proved by Dziubański \cite[Theorem 4.5]{Dz1}.  His proof uses nilpotent Lie group techniques for which it is important that $V$ is a polynomial. Moreover, in the particular case of $V(x)=|x|^2$ Bongioanni and Torrea \cite[Lemma 3]{BT} proved the $L^p(\RR^d),$ $1\leqslant p\leqslant \infty$, boundedness of $R_V^a$ for all $a>0$ by using explicitly the Mehler formula. Our proofs of \thref{thm:exLinfint,thm:exL1int} do not require explicit formulas  and the examples listed there are resistant to small perturbations; for instance, we may take $V(x)=\abs{x}^{\alpha}+E(x)$ with $\alpha>0,$ whenever the error term $E$ is a locally bounded function of a lower order than  $|x|^{\alpha}$ for large values of $|x|.$ 
	

The $L^{\infty}$ boundedness of $R^1_V$ was addressed by Urban and Zienkiewicz in \cite{UrZien1}. In \cite[Theorem 1.1]{UrZien1} the authors proved the $L^{\infty}(\RR^d)$ boundedness of $R^1_V$ under the assumption that $V$ is a non-negative polynomial satisfying a certain condition of C. Fefferman. This condition is of an algebraic nature. The estimates depend only on properties of the polynomial $V$ and are independent of the dimension. Recently, the first author proved a dimension-free $L^{\infty}$ bound for $R^{1/2}_V$ in the particular case of $V(x)=|x|^2$  and $L$  being the harmonic oscillator, see \cite[Theorem 8]{Ku1}. In fact it is proved there that the $L^{\infty}$ norm of  $R^{1/2}_{|x|^2}$ is less than $3.$ It is not clear whether one can prove dimension-free results on $L^{\infty}$  as in \cite{UrZien1} or \cite{Ku1} for $R^{1}_V$ or  $R^{1/2}_V$ for more general classes of potentials $V$. We hope to return to this topic in the near future.

It is perhaps noteworthy that in order to conclude the $L^p(\R^d),$ $p>2,$ boundedness of $R_V^{1/2}$, $R_V^1$ or $|\nabla L^{-1/2}|$ the results available in the literature require that $V$ satisfies at least a reverse H\"older condition. Such a $V$ must then be a doubling weight. This is not required in our approach, for instance $V(x)=\beta^{|x|}$ is clearly non-doubling yet \thref{thm:exLinfint,thm:exL1int} apply.

We shall now describe the structure of our paper. Section \ref{sec:def} starts with definitions of the objects appearing throughout the paper. Then we prove several interpolation results for the Riesz transforms $R_V^a,$ see \thref{thm:intLp,thm:intL1} and  \thref{cor:intLp}. As an application, in \thref{thm:genLp} we obtain $L^p$ boundedness of $R_V^a$ for general non-negative potentials  $V\in L_{\rm loc}^2$ within the  range  $1<p\leqslant 2,$ $0\leqslant a\leqslant 1/p.$  In Section \ref{sec:linf} we prove \thref{th:linf} which gives sufficient conditions for the $L^\infty$ boundedness of $R_V^a$ and then we apply it to prove \thref{thm:exLinfint}. Section \ref{sec:l1} is devoted to proving \thref{th:l1,th:l1a<1,th:l1ho} in which we present different conditions on $V$, $a$ and $p$ guaranteeing the $L^1$ boundedness of $R_V^a$ and as a corollary \thref{thm:exL1int} is proved.

\subsection*{Notation}
Throughout the paper for $1\leqslant p\leqslant \infty$ we denote by $L^p$  the $L^p(\R^d)$ space with respect to the $d$-dimensional Lebesgue measure.  For a function $f\in L^p$ we write $\|f\|_p \coloneqq \|f\|_{L^p(\R^d)}.$ Similar notation is also used for a bounded linear operator $T$ on $L^p;$ by $\|T\|_p$ we denote its norm. Although this is a slight collision of symbols it will cause no confusion later. For a Lebesgue-measurable subset $A\subseteq \R^d$ we denote by $|A|$ its Lebesgue measure. We say that $f$ is a finitely simple function if it is a simple function supported in a compact subset of $\R^d.$ Such functions are clearly dense in $L^p,$ $1\leqslant p<\infty.$  For a set $A$ we denote by $\ind{A}$ its characteristic function. The symbol $\ind{}$ stands for the constant function $1.$ For $1\leqslant p\leqslant \infty$ we denote by $L^{p}_{\rm loc}$ the space of functions which are locally in $L^p.$  For $f\in L^1_{\rm loc}$ we denote by $\supp f$ its essential support. The space of smooth compactly supported functions on $\R^d$ is denoted by $C_c^{\infty}.$ For $x\in \R^d$ and $r>0$  we denote by $B(x,r):=\{y\in \R^d\colon |x-y|\le r\}$ the closed Euclidean ball of radius $r$.

The symbol $C_{\square }$ denotes a non-negative constant that depends only on the parameter $\square.$ The exact value of $C_{\square }$ may change from one occurrence to another. We write $C$ without subscript when the constant is universal in the sense that it may only depend on the dimension $d$ or on the parameter of the Riesz transform $a>0.$

It will be convenient to introduce an asymptotic notation. For two non-negative quantities
$A, B$ we write $A \lesssim B$ ($A \gtrsim B$) if
there is an absolute constant $C>0$  such that $A\leqslant CB$ ($A\geqslant C B$). 
We
will write $A \approx B$ when $A \lesssim B$ and
$A\gtrsim B.$ In particular, if $A=A(x)$ and $B=B(x)$ are two non-negative functions on $\R^d$ then by $A\lesssim B$ we mean that
$A(x)\leqslant C B(x)$ for almost all $x\in \R^d;$ similar convention is applied to the symbols $\gtrsim$ and $\approx.$ We say that a function $B\colon \R^d \to [0,\infty)$ controls a function $A\colon \R^d\to  [0,\infty)$ globally if there exists a compact set $F$ such that $A(x)\leqslant B(x)$ for almost all $x\not\in F.$ In this case we write $A\leg B.$ Similarly, we say that any of the conditions $A\lesssim B,$ $A\gtrsim B$ or $A\approx B$ holds globally if there exists a compact set $F$ such that  $A(x)\lesssim B(x),$ $A(x)\gtrsim B(x)$ or $A(x)\approx B(x),$ respectively, hold for almost every $x\not\in F.$ In this case we write, respectively, $A\lesg B,$ $A\gtrsg,$ and $A \ag B.$

For a random variable $X$ defined on a probability space $(\Omega, \mathcal{F}, \P)$ and $A \subseteq \R$ we denote $\P(X \in A) \coloneqq \P\left( \left\{ \omega \in \Omega: X(\omega) \in A \right\} \right)$. We abbreviate \emph{almost everywhere} and \emph{almost every} to \emph{a.e}.

\subsection*{Acknowledgments} We are most grateful to the anonymous referee for the careful reading of the paper and helpful suggestions which helped us to improve considerably
the manuscript.



\section{Definitions and general results on $L^p,$ $1\le p<\infty$}
\label{sec:def}

The main goal of this section is to define the Riesz transforms $R_V^a,$ $a>0,$ on $L^p$  and to prove $L^p$ boundedness results for these operators valid for general classes of non-negative potentials $V.$ Throughout this section we take $1\le p<\infty.$ The case of $p=\infty$ is addressed in the next section.

Our general definition on $L^p$ will be based on semigroups related to $-\frac12\Delta+V$ that are given by the spectral theorem. Let $V\in L^1_{\rm loc}$ be an a.e. non-negative potential. This assumption is in force throughout the paper even if this is not stated explicitly. Whenever we write $V(x)$ we mean the value at $x$ of a particular representative of the equivalence class of $V$ in the space $L_{\rm loc}^1$. The same is true for any expression in which similar ambiguity may arise. We follow closely the approach in \cite[Section 3]{AuBA} (see also \cite{davies}) and define the Schr\"{o}dinger operator $L$ via quadratic forms. Consider the sesquilinear form
\begin{equation} \label{eq:Q}
	Q(u, v) = \int_{\R^d} \tfrac{1}{2} \innprod{\nabla u}{\nabla v} + Vu\overline{v}
\end{equation}
on the domain
\[
	\dom(Q) = \{f \in L^2: \nabla f \in L^2 \text{ and } V^{1/2}f \in L^2\},
\]
where $\nabla f$ denotes the distributional gradient of $f$.
We equip the domain with the norm
\[
	\norm{f}_V = \left( \norm{f}_2^2 + \tfrac{1}{2} \norm{\nabla f}_2^2 + \norm{V^{1/2}f}_2^2 \right)^{1/2},
\]
which turns it into a Hilbert space with $C_c^\infty(\R^d)$ as a dense subspace. Since $Q$ is bounded below and non-negative, there is a unique positive self-adjoint operator $L$ such that
\[
	\innprod{Lu}{v} = Q(u, v), \quad u \in \dom(L), \, v \in \dom(Q)
\]
and its square root $L^{1/2}$, defined on $\dom(L^{1/2}) = \dom(Q)$, satisfies
\begin{equation} \label{eq:L12}
	\norm{L^{1/2} f}_2^2 = \tfrac{1}{2} \norm{\nabla f}_2^2 + \norm{V^{1/2} f}_2^2, \quad f \in C_c^\infty(\R^d).
\end{equation}
By \cite[Section 3]{AuBA} the semigroup $e^{-tL}$ is positivity-preserving and pointwise dominated by the heat semigroup, hence it is a contraction on $L^p$ for $1 \leqslant p \leqslant \infty$.

Let $a>0.$  For $f\in L^p,$ $1\leqslant p<\infty$, and $\varepsilon>0$ we define
\begin{equation}
\label{def:Lalime}
(L+\varepsilon I)^{-a}f  \coloneqq \frac{1}{\Gamma(a)}\int_0^{\infty} e^{-tL} f\,t^{a-1}e^{-\varepsilon t}\,dt,
\end{equation}  
Since  the semigroup $e^{-tL}$ is a strongly continuous semigroup of contractions on $L^p$, the integral in \eqref{def:Lalime} is well defined as a Bochner integral on $L^p$. It is also not hard to see that for $f\in L^2$ the operator defined by \eqref{def:Lalime} coincides with $(L+\varepsilon I)^{-a}$ given by the spectral theorem. Moreover, if $f$ is an a.e. non-negative function in $L^p$ then 
\begin{equation}
\label{def:Lalim}
L^{-a}f(x)  \coloneqq \lim_{\varepsilon\to 0^+}\frac{1}{\Gamma(a)}\int_0^{\infty} e^{-tL} f(x)\,t^{a-1}e^{-\varepsilon t}\,dt,
\end{equation}  
exists $x$-a.e.\ as a monotone pointwise limit being finite or infinite.
In either case
\begin{equation}
\label{def:La}
L^{-a}f(x) =\frac{1}{\Gamma(a)}\int_0^{\infty} e^{-tL} f(x)\,t^{a-1}\,dt,
\end{equation} 
by the monotone convergence theorem.
For $a>0$ and a non-negative function $f\in L^p$ we let
\begin{equation}
\label{eq:Rav}
R_V^a f(x) \coloneqq V^a(x)L^{-a}f(x),\qquad x\in \R^d.
\end{equation}
This is well defined $x$-a.e.\ though possibly equal to infinity. Additionally,  for $a=0$ we set $R_V^0$ to be the identity operator.

\begin{df}
	\label{def:Ravb}
	Let $1\leqslant p<\infty$ and $a>0.$ We say that the Riesz transform $R_V^a$ is bounded on $L^p$ if there is a constant $C>0$ such that
	\begin{equation}
	\label{eq:Ravb}
	\|R_V^a f\|_p\leqslant C \|f\|_{p},
	\end{equation}   
	for all non-negative finitely simple functions $f\in L^p.$ 
\end{df}

Note that if $R_V^a$ is bounded on $L^p$, then for each finitely simple function $f$ the quantity $R_V^a |f|$ given by \eqref{eq:Rav} is finite for a.e.\ $x\in \R^d.$ Since $|e^{-tL}f|\leqslant e^{-tL}|f|$  we see that in this case 
\begin{equation*}
V^a(x)\int_0^{\infty} e^{-tL} f(x)\,t^{a-1}\,dt
\end{equation*}
is finite $x$-a.e.. Thus, whenever $R_V^a$ is bounded on $L^p$ the integral above is a natural definition of  $R_V^a f,$ first for finitely simple functions  and then, by density, for arbitrary functions in $L^p$.

Using Stein's complex interpolation theorem and functional calculus for symmetric contraction semigroups \cite{Hanonsemi} we now prove an interpolation result for the operators $R_V^a.$ Similar method was applied in \cite[Section 6]{AuBA}. There the authors proved the $L^p$ boundedness of $R_{V}^{1/2}$ for $1<p<2(q+\varepsilon)$ by using Stein's complex interpolation theorem together with the $L^p$ boundedness of $R_V^1.$ They considered non-negative potentials belonging to a reverse H\"older class $B_q$.    

\begin{theorem}
	\thlabel{thm:intLp}
	Let $0\leqslant a_0<a_1.$  Assume that $V\in L^1_{\rm loc}$ is an a.e. non-negative potential such that $R_V^{a_0}$ is bounded on $L^{p_0}$ and $R_V^{a_1}$ is bounded on $L^{p_1}$ for some $p_0,p_1\in (1,\infty).$ Then, $R_V^a$ is bounded on $L^p$ for every $p$ and $a$ such that $1/p=\theta/p_0+(1-\theta)/p_1$ and $a=\theta a_0 +(1-\theta)a_1$ with some $\theta\in (0,1).$
\end{theorem}

\begin{proof}
	Let $\varepsilon>0$ and denote $F(\varepsilon) \coloneqq \{x\in \R^d\colon \varepsilon<V(x)<\varepsilon^{-1}\}.$  It is enough to justify that 
	\[
		R^{a,\varepsilon}f(x) \coloneqq (\ind{F(\varepsilon)}V^{a})(x)\cdot \frac{1}{\Gamma(a)}\int_0^{\infty} e^{-tL}f(x)\,t^{a-1}e^{-\varepsilon t}\,dt,
	\]	
	satisfies for all simple functions $f$ the bound
	\begin{equation}
	\label{eq:Raeuni}
		\|R^{a,\varepsilon}f\|_p\leqslant C \|f\|_{p},
	\end{equation}
	uniformly in $\varepsilon>0$ and with $C>0$ being a constant. Indeed, if \eqref{eq:Raeuni} holds, then taking $\varepsilon\to 0^+$ we obtain the $L^p$ boundedness of $R_V^a,$ first (with the aid of monotone convergence theorem) for non-negative simple functions and then for all functions in $L^p.$
	
	Thus, in the remainder of the proof we fix $\varepsilon>0$ and focus on justifying \eqref{eq:Raeuni}. Denote $S=\{z\in \C\colon a_0 < \Rea z< a_1\}.$ Then, for $z\in \overline{S}$ and $\varepsilon>0$ the function $m_z^{\varepsilon}(\la)=(\la+\varepsilon)^{-z}$ is a bounded function on $[0,\infty),$ hence, by the spectral theorem  $(L+\varepsilon I)^{-z}$ is well defined as a bounded operator on $L^2.$ We let 
	\begin{equation}
	\label{eq:Tzdef}
		T_z f \coloneqq (\ind{F(\varepsilon)}V^{z})\cdot (L+\varepsilon I)^{-z}f,\qquad f\in L^2.
	\end{equation}
	Since $(L+\varepsilon I)^{-b}$ given by the spectral theorem coincides with
	\[
		\frac{1}{\Gamma(b)}\int_0^{\infty} e^{-tL} f\,t^{b-1}e^{-\varepsilon t}\,dt,
	\]
	for every $b>0,$ we have
	\[
		R^{b,\varepsilon} f =T_b f,\qquad f\in L^2.
	\]
	Thus, in order to justify \eqref{eq:Raeuni} it suffices to prove a uniform in $\varepsilon>0$ bound  for the $L^p$ norm of $T_a.$

	This will be achieved by Stein's complex interpolation theorem. Note first that for $f,g$ being finitely simple functions the pairing
	\[
		h(z)=\langle T_z f,g\rangle,\qquad z\in \overline{S}, 
	\]
	gives a function which is holomorphic in $S$. To see this observe that \eqref{def:Lalime} still holds  with complex $a \in S.$ Combining this observation with the definition \eqref{eq:Tzdef} of $T_z$ it is easy to see that $h$ is indeed holomorphic. Additionally, the spectral theorem implies the bound
	\begin{equation}
	\label{eq:growthin}
		|h(z)|\leqslant C(\varepsilon,f,g),
	\end{equation}
	valid for $z\in \bar{S}.$  Altogether  $\{T_z\}_{z\in \bar{S}}$ is an analytic family of operators of admissible growth.  
	
	It remains to bound the operator $T_z$ for $\Rea z=a_0$ and $\Rea z=a_1;$ this is the place where we use the assumptions on $R^{a_j}_V.$ Writing, for $z=a_j+i\tau,$ $\tau\in \RR,$ $j=0,1,$ 
	\[
		T_z=(\ind{F(\varepsilon)}V^{z})\cdot (L+\varepsilon I)^{-z}=(\ind{F(\varepsilon)}V^{i\tau})T_{a_j}(L+\varepsilon I)^{-i\tau}
	\]
	we see that
	\begin{equation}
	\label{eq:Tzes}
		\|T_z\|_{p_j}\leqslant \|T_{a_j}\|_{p_j}\|(L+\varepsilon I)^{-i\tau}\|_{p_j}.
	\end{equation}
	Since $(L+\varepsilon I)$ generates a symmetric contraction semigroup and $p_j\in(1,\infty),$ by e.g.\ \cite{Hanonsemi} the imaginary powers $(L+\varepsilon I)^{-i\tau}$ satisfy
	\begin{equation}
	\label{eq:Imeps}
		\|(L+\varepsilon I)^{-i\tau}\|_{p_j}\lesssim e^{\pi|\tau|/2},
	\end{equation}
	uniformly in $\varepsilon>0.$ Moreover, we have
	\[
		|T_{a_j} (f)(x)|=|R^{a_j,\varepsilon}f(x)|\leqslant R^{a_j}_V|f|(x),\qquad x\in \R^d.
	\]
	Thus, coming back to \eqref{eq:Tzes} and using our assumptions on the $L^{p_j}$ boundedness of $R^{a_j}_V$ we obtain, for $z=a_j+i\tau,$ $j=0,1,$
	\[
		\|T_z\|_{p_j}\lesssim  e^{\pi|\tau|/2},\qquad \tau\in\R.
	\]
	
	In conclusion, applying Stein's complex interpolation theorem, see e.g.\ \cite[Theorem 1.3.7]{grafakos}, we obtain the $L^p$ boundedness of $R_V^a$.
\end{proof}

\thref{thm:intLp} immediately leads to the following corollary.
\begin{cor}
	\thlabel{cor:intLp}
	Let $a_0\geqslant 0,$ $a_1\geqslant 0,$ and  assume that both $R_V^{a_1}$ and $R_V^{a_2}$ are bounded on $L^{p}$ for some $1 < p<\infty.$ Then $R_V^a$ is bounded on $L^p$ for every $a_0\leqslant a\leqslant a_1.$
\end{cor}
\begin{proof}
	We apply \thref{thm:intLp} with $p_0=p_1=p$.
\end{proof}

It is straightforward to see that the Riesz transform $R_V^{1/2}$ is bounded on $L^2.$ Using \thref{cor:intLp} we now extend the $L^2$ boundedness to the operators $R_V^{a}$ with $0\leqslant a\leqslant \frac{1}{2}.$ 
\begin{pro}
	\thlabel{pro:L2}
	Let $V\in L^1_{\rm loc}(\RR^d)$ be an a.e. non-negative potential. If $0\leqslant a\leqslant \frac{1}{2} $, then $R^{a}_V$ extends to a contraction on $L^2(\RR^d).$
\end{pro}
\begin{proof}
	
	By formula \eqref{eq:L12} we have
	\begin{equation}
		\label{eq:V12Lin}
	\norm{V^{1/2} f}_2 \leqslant \norm{L^{1/2} f}_2,\qquad f \in C_c^\infty;
	\end{equation}
here $L^{1/2}$ is the self-adjoint operator with domain $\dom(L^{1/2})=\dom(Q),$ while $Q$ is the sesquilinear form given by \eqref{eq:Q}.
	Using the fact that self-adjoint operators are closed we get $\dom(L^{1/2}) \subseteq \dom(V^{1/2})$ and 
	\begin{equation} \label{eq:VL}
	\norm{V^{1/2} f}_2 \leqslant \norm{L^{1/2} f}_2,\qquad f \in \dom(L^{1/2}).
	\end{equation}
	For each fixed $\varepsilon>0$ the operator $(L+\varepsilon I)^{-1/2}$ is bounded on $L^2$ by the spectral theorem. Taking $f = (L+\varepsilon I)^{-1/2} g$ with $g \in L^2$ in \eqref{eq:VL} we get
	\begin{equation}
	\label{eq:VLeps}
	\norm{V^{1/2}(L+\varepsilon I)^{-1/2} g}_2 \leqslant \norm{L^{1/2}(L+\varepsilon I)^{-1/2}g}_2, \qquad g \in L^2.
	\end{equation}
	If $g$ is a non-negative function on $L^2$ then by definitions \eqref{def:Lalime}, \eqref{eq:Rav} and the monotone convergence theorem 
	we have
	$\lim_{\varepsilon\to 0^+} 	\norm{V^{1/2}(L+\varepsilon I)^{-1/2} g}_2=\|R_V^{1/2}g\|_2.$ The right-hand side of \eqref{eq:VLeps} converges to $\|g\|_2$ as $\varepsilon\to 0^+$ by the spectral theorem. Therefore we justified that $\|R_V^{1/2}g\|_2\leqslant \|g\|_2$ for non-negative $g\in L^2.$ This implies that $R_V^{1/2}$ is a contraction on $L^2.$
	
	At this stage an application of \thref{cor:intLp} shows that $R_V^{a}$ is bounded on $L^2$ whenever $0\leqslant a \leqslant 1/2.$ The contractivity of $R_V^a$ is not a direct consequence of the corollary. However, it is easy to justify once we follow the proof of \thref{thm:intLp} and enhance inequality \eqref{eq:Imeps} to
	\[
		\|(L+\varepsilon I)^{-i\tau}\|_{2}\leqslant 1,\qquad \tau\in \R.
	\]
	We omit details here.
	
\end{proof}

When $p_0=1$ we have a slightly weaker variant of \thref{thm:intLp} with the restriction $a_0, a_1>0.$ This is due to the unboundedness of the imaginary powers $L^{i\tau},$ $\tau\in\R,$ on  $L^1.$

\begin{theorem}
	\thlabel{thm:intL1}
	Let $a_0>0$ and $a_1>0.$  Assume that $V\in L^1_{\rm loc}$ is such that $R_V^{a_1}$ is bounded on $L^{p_1}$ for some $p_1\in [1,\infty)$ and $R_V^{a_0}$ is bounded on $L^1.$ Then, $R_V^a$ is bounded on $L^p$ for every $p$ and $a$ such that $1/p=\theta+(1-\theta)/p_1$ and $a=\theta a_0 +(1-\theta)a_1$ with some $\theta\in (0,1).$
\end{theorem} 

\begin{proof}
	The proof is similar to that of \thref{thm:intLp}. For $\varepsilon>0$ we define the sets $F(\varepsilon)$ and the operators $R^{a,\varepsilon}$ as in that proof. Once again it suffices to justify \eqref{eq:Raeuni}. 
	
	Assume without loss of generality that $a_0 < a_1$, let  $S=\{z\in \C\colon a_0 < \Rea z< a_1\}$ and define the family of operators $\{T_z\}_{z\in \bar{S}}$ as in \eqref{eq:Tzdef}. Since this time $a_0>0$ the formula
	\begin{equation}
	\label{eq:Tzdef'}
	T_z f=(\ind{F(\varepsilon)}V^{z})\cdot\frac{1}{\Gamma(z)}\int_0^{\infty}e^{-tL}f\, t^{z-1}e^{-\varepsilon t}\,dt,\qquad f\in L^2,
	\end{equation}
	holds for $z\in \bar{S}.$ Moreover, $\{T_z\}_{z\in S}$ is a family of analytic operators of admissible growth; this can be justified as in the proof of \thref{thm:intLp}. Hence, in order to apply Stein's complex interpolation theorem it remains to bound $\|T_z\|_{p_j}$ for $z=a_j+i\tau,$ $j=0,1.$ Using \eqref{eq:Tzdef'} and the asymptotics for the Gamma function $|\Gamma(a_j+i\tau)|\approx |\tau|^{a_j-1/2}e^{-\pi |\tau|/2},$ see \cite[5.11.9]{nist}, we obtain the pointwise bound
	\[
	|T_z f(x)|\lesssim e^{\pi |\tau| } (\ind{F(\varepsilon)}V^{a_j})(x)\cdot \int_0^{\infty}e^{-tL}|f|(x)\, t^{a_j-1}e^{-\varepsilon t}\,dt\lesssim  e^{\pi |\tau| } R_V^{a_j}|f|(x),
	\]
	valid for $z=a_j+i\tau,$ $j=0,1.$ Hence, the $L^{1}$ boundedness of $R_V^{a_0}$ together with the $L^{p_1}$ boundedness of $R_V^{a_1}$ give
	\[
	\|T_z\|_{1}\lesssim  e^{\pi |\tau| },\qquad z=a_0+i\tau,\quad \tau\in\R,
	\]
	and
	\[
	\|T_z\|_{p_1}\lesssim  e^{\pi |\tau| },\qquad z=a_1+i\tau,\quad \tau\in\R.
	\]
	Thus, using Stein's complex interpolation theorem we complete the proof.

\end{proof}

%

Analogously to the $L^2$ case one particular Riesz transform $R_V^1$ is always bounded on $L^1,$ see \cite[Theorem 4.3]{AuBA} and \cite{GallMor, katoLp}. Interpolating this result with \thref{pro:L2} we obtain the following theorem.

\begin{theorem}
	\thlabel{thm:genLp}
	Let  $V\in L^1_{\rm loc}$ and take $p\in (1,2].$ Then for all $0\leqslant a \leqslant 1/p$ the Riesz transform $R_V^a$ is bounded on $L^p.$
\end{theorem} 
\begin{proof}
	The $L^2$ boundedness of $R_V^{1/2}$ is guaranteed by \thref{pro:L2}. The $L^1$ boundedness of $R_V^1$ is justified in \cite[Theorem 4.3]{AuBA}. Hence, \thref{thm:intL1} gives the $L^p$ boundedness of $R_V^a$ whenever $a=\theta+(1-\theta)/2=1/p.$ Finally, \thref{cor:intLp} extends the boundedness on $L^p$ to $0\leqslant a\leqslant 1/p.$ 
	
\end{proof}

\section{Definitions  and a counterexample on $L^{\infty}$ } 
Here the approach from the previous section is invalid since $e^{-tL}$ does not necessarily extend to a strongly continuous semigroup on $L^{\infty}.$ 
However, for certain classes of potentials the operator $	e^{-tL},$ $t>0,$ can be also expressed by the celebrated Feynman--Kac formula  
\begin{equation} \label{feynman-kac}
e^{-tL} f(x) = \E_x \left[ e^{-\int_0^t V(X_s) ds} f(X_t) \right],\qquad f\in L^p,
\end{equation}
where $1\leqslant p<\infty.$
The expectation $\E_x$ is taken with regards to the Wiener measure of the standard $d$-dimensional Brownian motion $\{X_s\}_{s>0},$ starting at $x\in \R^d;$  here $X_s=(X_s^1,\ldots,X_s^d).$  
Since the potential $V$ is a.e. non-negative, identity \eqref{feynman-kac} is true whenever $V\in L^2_{\rm loc}$ belongs to the local Kato class $K_d^{\rm loc}$. This follows for example from \cite[Remark 4.14]{sznitman} once we recall that for $V\in L^2_{\rm loc}$ the operator $-\Delta/2+V$ is essentially self-adjoint on $C_c^{\infty}$, hence its Friedrichs extension is its unique self-adjoint extension. We will not need the definition of the local Kato class in our paper; for our purpose it is important to note that $L^{q}_{\rm loc}\subseteq K_d^{\rm loc}$ whenever $q\geqslant 1$ satisfies $q>d/2,$   see \cite[Lemma 4.105]{LHB}. Therefore \eqref{feynman-kac} is true for $V\in L^{q}_{\rm loc}$ whenever $q>d/2$ and $q\geqslant 2$, in particular for $V\in L^{\infty}_{\rm loc}.$ The right-hand side of \eqref{feynman-kac} makes sense also for $f\in L^{\infty},$ see \cite[Section 4.2.4]{LHB}, which leads us to define for $t>0$
\begin{equation} \label{feynman-kac1}
e^{-tL} f(x) \coloneqq \E_x \left[ e^{-\int_0^t V(X_s) ds} f(X_t)  \right], \qquad f\in L^{\infty}.
\end{equation} 
To deal with measurability questions we need a technical lemma on the continuity of $e^{-tL} f.$
\begin{lemma}
	\thlabel{lem:contfk}
	Assume that $q>d/2$ and $q\geqslant 2$ and let  $V\in  L^{q}_{\rm loc}$  be an a.e. non-negative potential. Then for all $f\in L^{\infty}$ the function  $e^{-tL} f(x)$ given by \eqref{feynman-kac1} is jointly continuous in $(t,x)\in (0,\infty)\times \R^d.$ In particular $e^{-tL}(\ind{})(x)$ is jointly continuous in $t$ and $x$.
\end{lemma}
\begin{proof}
	Since $L^{q}_{\rm loc}\subseteq K_d^{\rm loc}$ it follows from \cite[Proposition 3.5]{sznitman} that $e^{-tL}$ is an integral operator with its kernel $K_t(x,y)$ being a jointly continuous functions of $(t,x,y).$ 
	Since $V\geqslant 0$ we also have $K_t(x,y)\leqslant  (2\pi t)^{-d/2}\exp(|x-y|^2/(2t))$ and therefore for each $N>0$ it holds 
	\begin{equation}
	\label{eq:contf1}
	\int_{|x-y|> N}K_t(x,y)|f(y)|\,dy \leqslant \pi^{-d/2} \|f\|_{\infty}  \int_{|w|\geqslant N/(\sqrt{2t})} e^{-|w|^2}\,dw.
	\end{equation}
	Consider now $(t,x)\to (t_0,x_0)$ and let $\varepsilon >0$ be arbitrarily small. Splitting
	\begin{align*}
	e^{-tL}f(x)=\int_{|x-y|\leqslant N} K_t(x,y)f(y)\,dy + \int_{|x-y|> N}K_t(x,y)f(y)\,dy
	\end{align*} 
	and using \eqref{eq:contf1} we see that for $N=N(\varepsilon)$ large enough holds
	\begin{align*}
	\left|	e^{-tL}f(x)-\int_{|x-y|\leqslant N} K_t(x,y)f(y)\,dy\right|\leqslant \varepsilon,
	\end{align*}
	uniformly in $t_0/2<t<2t_0$ and  $\abs{x-x_0}<1$. Moreover, for such $(t,x)$ we see that $C\|f\|_{L^{\infty}}\ind{|y|\leqslant N+|x_0|+1}$ is an integrable majorant of $\ind{|x-y|\leqslant N}K_t(x,y)f(y).$ Thus, using Lebesgue's dominated convergence theorem we obtain
	\[
	\limsup_{(t,x)\to (t_0,x_0)}|e^{-tL}f(x)-e^{-t_0L}f(x_0)|\leqslant 2\varepsilon.
	\]
	Since $\varepsilon>0$ was arbitrary this completes the proof.
\end{proof}

Now, take $a>0$ and let $V\in L^{\infty}_{\rm loc}$ be an a.e. non-negative potential. For a non-negative function $f\in L^{\infty}$ we define the Riesz transform $R_V^a$ by
\begin{equation}
\label{eq:Ravinf}
R_V^a f(x)=V^a(x)\cdot\frac{1}{\Gamma(a)}\int_0^{\infty} \E_x \left[ e^{-\int_0^t V(X_s) ds}f(X_t)  \right]\,t^{a-1}\,dt,\qquad f\in L^{\infty}.
\end{equation}
Note that by \thref{lem:contfk} the function $R_V^af(x)$
is then a measurable function on $\R^d$ possibly being infinite for some $x.$ 
Moreover, by \eqref{feynman-kac} the $L^{\infty}$ definition \eqref{eq:Ravinf} coincides with the $L^p$ definition \eqref{eq:Rav} whenever $f$ is a finitely simple function.

Since the semigroup is positivity preserving we have
\begin{equation}
\label{eq:etL(1)comp}
|e^{-tL}f(x)|\leqslant e^{-tL}(\|f\|_{\infty}\ind{})(x)=\|f\|_{\infty}e^{-tL}(\ind{})(x),\qquad f\in L^{\infty},
\end{equation}
which leads to the following definition of the $L^{\infty}$ boundedness of $R_V^a.$

\begin{df}
	\label{def:Ravbinf}
	We say that the Riesz transform $R_V^a$ is bounded on $L^{\infty}$ if 
	\begin{equation}
	\label{eq:Ravbinf}
	\|R_V^a (\ind{})\|_{\infty}<\infty.
	\end{equation}    
\end{df}
\noindent Note that if \eqref{eq:Ravbinf} holds, then for every $f\in L^{\infty}$ by \eqref{eq:etL(1)comp} we have $|R_V^a (f)(x)|\leqslant \|f\|_{\infty}R_V^a(\ind{})(x)$ so that 
\begin{equation}
\label{eq:Ravbinfgen}
\|R_V^a (f)\|_{\infty} \leqslant C\|f\|_{\infty},\qquad f\in L^{\infty}
\end{equation}   
with $C=\|R_V^a (\ind{})\|_{\infty}.$

Since 
\begin{equation}
\label{eq:Rav1}
R_V^a(\ind{})(x)=V^a(x)\cdot\frac{1}{\Gamma(a)}\int_0^{\infty} e^{-tL}(\ind{})(x) \, t^{a-1}\,dt
\end{equation}
it is apparent that in order for $R_V^a$ to be finite a.e. on $\supp V$ the monotone function $t\mapsto e^{-tL}(\ind{})(x)$ must converge to $0$ as $t\to \infty.$ This however is not always the case.
\begin{pro}
	\thlabel{pro:cexlinf}
	Let $d\geqslant 3$ and let  $V$ be a non-negative potential on $\R^d$ which is compactly supported and not identically equal to zero. Assume that $V\in L^{q}(\R^d)$ with $q>d/2$ and $q \geqslant 2$. Then, for any $a>0$ we have $R_V^a(\ind{})(x)=\infty$ for $x$ such that $V(x) \neq 0$. In particular $R_V^a$ is unbounded on $L^{\infty}.$  
\end{pro}
\begin{proof}
	Fix $a>0$. For $x\in \R^d$ we let $w(x)=\lim_{s\to \infty}e^{-sL}(\ind{})(x).$ From \cite[Lemma 2.4]{DzZi1} there exist a constant $\delta>0$ such that $\delta <w(x)\leqslant 1$ uniformly in $x\in \R^d.$ Since by the semigroup property $w(x)=e^{-tL}(w)(x)$ for any $t>0,$ we see that $e^{-tL}(\ind{})\geqslant e^{-tL}(w)(x)\geqslant \delta$ uniformly in $x\in \R^d.$ Consequently, the integral $\int_0^{\infty}e^{-tL}(\ind{})(x)\,t^{a-1}\,dt$ is infinite for a.e.\ $x$ and so is $R_V^a(\ind{})(x)$ as long as $V(x) \neq 0$.
\end{proof}

The definition below is meant to guarantee the $x$-a.e. finiteness of $R_V^af(x).$
\begin{df}
	\label{ed}
	Let $V\in L^{\infty}_{\rm loc}$ be an a.e. non-negative potential and let $\delta>0.$ We say that the semigroup $e^{-tL}$ has an exponential decay of order $\delta$ ($\ed(\delta)$ for short) if there exists a constant $C>0$ such that
	\begin{equation}
	\tag{$\ed(\delta)$}
	\label{eq:ed}
	\|e^{-tL}(\ind{})\|_{\infty}\leqslant C e^{-\delta t},\qquad t>0.
	\end{equation}
\end{df}
The assumption $(ED(\delta))$ implies $\abs{R_V^a f(x)} \leqslant C\delta^{-a} V^a(x) \norm{f}_\infty$ $x$-a.e.. Note, however, that this may not be enough to conclude that $\norm{R_V^a (\ind{})}_\infty < \infty$.

\section{$L^\infty$ boundedness for classes of potentials} \label{sec:linf}
	
	Throughout this section we assume that $V \in L_{\rm loc}^\infty$. Here our goal is to estimate the $L^\infty$ norm of $R_V^a$ for classes of potentials $V$. As mentioned in Definition \ref{def:Ravbinf} this is the same as estimating $\norm{R_V^a (\ind{})}_{\infty}$ with $R_V^a(\ind{})$ defined by \eqref{eq:Rav1}.
	
	Before we dive into details, we prove a general result concerning the $L^{\infty}$ decay of the semigroup $e^{-tL}$ defined in \eqref{feynman-kac1}. We will use \thref{lem21} below to prove the $L^\infty$ and $L^1$ boundedness of $R_V^a$ for concrete examples of potentials $V$ in \thref{thm:exLinfint,thm:exL1int}. Here $\pi$ denotes a $(d-1)$-dimensional hyperplane in $\R^d$. For $N>0$ we let $P$ be the strip $$P =P_N:= \{x \in \R^d: \dist(x, \pi) \leqslant N\}$$ and set $\chi=\ind{P}.$
	
	\begin{lemma}
	\thlabel{lem21}
		Let $N>0$ and assume that the potential $V \in L_{\rm loc}^\infty$ is uniformly positive outside the strip $P_N.$ More precisely we assume that $V$ is non-negative a.e.\ and that there is $c > 0$ such that $V(x) \geqslant c$ for a.e.\ $x$ satisfying $\dist(x, \pi) > N$. Then the semigroup $e^{-tL}$ has $ED(\delta)$ with $\delta = \frac{1}{2}\min\left( c, \frac{1}{8N^2} \right).$ More precisely, there is a universal constant $C>0$ such that for $t>0$ and $x\in \R^d$ it holds
		\[
		e^{-tL}(\ind{})(x) \leqslant C\, e^{-\delta t}.
		\]
	\end{lemma}
	
	To prove the above lemma we will need an auxiliary fact. \thref{lem22} below can be deduced from \cite[Lemma 4.105]{LHB}. For the sake of completeness we give a more direct proof below.  

	\begin{lemma}
	\thlabel{lem22}
	For all $k>0,$ $t>0,$ and $x\in \R^d$ we have
	\begin{equation}
		\label{eq:lem22}
			\E_x \left[ e^{2\int_0^t k\chi(X_s) ds} \right] \leqslant C\, e^{8N^2k^2 t},
	\end{equation}
		where $C>0$ is a universal constant.
	\end{lemma}

	\begin{proof}

		We prove this fact in the case $\pi = \{0\} \times \R^{d-1}$ and $P = [-N, N] \times \R^{d-1}$. The general result follows from the invariance of Brownian motion under orthogonal transformations (see \cite[p. 5]{port_stone}) and the fact that the bound is independent of $x$. Since in this case $\chi(X_s) = \ind{[-N,N]}(X_s^1)$ it suffices to prove the lemma in the dimension $d=1.$ In particular in the proof we take $x\in \R.$
		
		The main tool of our proof is the local time of Brownian motion defined for $y\in \R$ in the one-dimensional case as
		\[
		L_t(y) = \lim_{\varepsilon \to 0^+} \frac{1}{2\varepsilon} \int_0^t \ind{[y-\varepsilon, \, y+\varepsilon]} (Y_s) \, ds,
		\]
		where $\{Y_s\}_{s>0}$ is the standard one-dimensional Brownian motion starting at $0.$
		It is well known that
		\[
		\int_0^t f(Y_s) \, ds = \int_\R f(y) L_t(y)\, dy
		\]
		for any locally integrable function $f$, see \cite[(5.4)]{borodin}. In particular, we have
		\begin{equation} \label{eq:loctime}
			\int_0^t \ind{[-N-x, \, N-x]}(Y_s) \, ds = \int_{-N-x}^{N-x} L_t(y)\, dy.
		\end{equation}
The law of $L_t(y)$ was computed by  Tak\'{a}cs \cite{takacs}. From a paper of Doney and Yor \cite{doney_yor}, see the last identity in Section 3 on p.\ 277 (with $\mu=0$ and $x=y$) and  \cite[eq.\ (1.4)]{doney_yor}, it follows that the distribution of $L_t(y)$ is given by
	\[
		c_{y,t} \delta_0 + f_{y,t}(z) \, dz
	\]
	on $[0, +\infty)$, where $\delta_0$ denotes the Dirac measure at $0,$
	\begin{equation}
		\label{eq:denfxt}
		f_{y,t}(z) = \frac{\sqrt{2}}{\sqrt{\pi t}} e^{-\frac{(|y|+z)^2}{2t}}, \qquad y \in \RR,\quad  z > 0,
	\end{equation}
	and $c_{y,t}<1$ is a normalizing constant which value is irrelevant for us.
		 
Using \eqref{eq:loctime} and Jensen's inequality for $x\in \RR$ we obtain
	\begin{align*}
		\E_{x} \left[ e^{2\int_0^t k\chi(X_s) ds} \right] &= \E_0 \left[ e^{2\int_{-N-x}^{N-x} kL_t(y) dy} \right] \leqslant \frac{1}{2N} \E_0 \left[ \int_{-N-x}^{N-x} e^{4NkL_t(y)}\, dy \right] \\
		&\leqslant \frac{1}{2N} \int_{-N-x}^{N-x} \left( 1 + \int_0^\infty e^{4Nkz} f_{y,t}(z) \, dz \right) \, dy \\
		&= 1 + \frac{1}{2N} \int_0^\infty e^{4Nkz} \int_{-N-x}^{N-x} f_{y,t}(z) \, dy \, dz
	\end{align*}
	The $1+$ term in the second line comes from the atom of the distribution of $L_t(y)$ at $z=0$. Since the function $y \mapsto f_{y,t}(z)$ is radially decreasing, we can change the limits of the inner integral to $[-N, N]$, possibly increasing its value. Thus, using \eqref{eq:denfxt} gives
	\begin{equation} \label{eq3}
		\begin{aligned}
			&1 + \frac{1}{2N} \int_0^\infty e^{4Nkz} \int_{-N-x}^{N-x} f_{y,t}(z) \, dy \, dz \leqslant 1 + \frac{1}{2N} \int_0^\infty e^{4Nkz} \int_{-N}^{N} f_{y,t}(z) \, dy \, dz \\
			&= 1 + \frac{\sqrt{2}}{N\sqrt{\pi t}} \int_0^\infty e^{4Nkz} \int_0^N e^{-\frac{(y+z)^2}{2t}} \, dy \, dz.
		\end{aligned}
	\end{equation}
	First we deal with the inner integral. Calculating it yields
	\[
		\int_0^N e^{-\frac{(y+z)^2}{2t}}\, dy = \sqrt{\frac{\pi t}{2}} \left( \erf\left( \frac{z+N}{\sqrt{2t}} \right) -\erf\left( \frac{z}{\sqrt{2t}} \right) \right).
	\]
	To estimate the expression above, note that $\erf'(y) = \frac{2e^{-y^2}}{\sqrt{\pi}}$, hence, by the mean value theorem 
	\[
		\erf\left( \frac{z+N}{\sqrt{2t}} \right) -\erf\left( \frac{z}{\sqrt{2t}} \right) = \frac{N}{\sqrt{2t}} \erf'(\theta),
	\]
	for some $\theta>z/(\sqrt{2t})$ and thus 
	\[
		\int_0^N e^{-\frac{(y+z)^2}{2t}} dy \lesssim  N e^{-\frac{z^2}{2t}}.
	\]
	Plugging the above estimate into \eqref{eq3}, we obtain
	\begin{align*}
		\E_{x} \left[ e^{2\int_0^t k\chi(X_s) ds} \right] &\lesssim 1+ \sqrt{\frac{2}{\pi t}} \int_0^\infty e^{4Nkz-\frac{z^2}{2t}} dz\\
	&  \lesssim  e^{8N^2k^2 t},
	\end{align*}
	which completes the proof of \thref{lem22}.
	\end{proof}
	
	Now we prove  \thref{lem21}. In the proof the quadratic dependence on $k$ on the right-hand side of \eqref{eq:lem22} will be crucial.

	\begin{proof}[Proof of \thref{lem21}]
		We want to make use of the assumption that the potential $V$ is uniformly positive outside the set $P$ together with the previous lemma. We achieve this by an appropriate application of the Cauchy--Schwarz inequality. 
		
		Recall that $\chi = \ind{P}$ and take $k\in(0,c].$ Since the potential $2(V + k\chi)$ is bounded below by $2k$ using Cauchy--Schwarz inequality we estimate
		\begin{align} \label{eq2}
			&e^{-tL}(\ind{})(x) = \E_x \left[ e^{-\int_0^t V(X_s) ds} \right] = \E_x \left[ e^{-\int_0^t V(X_s) + k\chi(X_s) ds} \, e^{\int_0^t k\chi(X_s) ds} \right]  \nonumber \\
			&\leqslant\left[\E_x  e^{-2\int_0^t V(X_s) + k\chi(X_s) ds} \right]^{1/2} \left[\E_x  e^{2\int_0^t k\chi(X_s) ds} \right]^{1/2} \nonumber \\
			&\leqslant e^{-kt } \E_x \left[ e^{2\int_0^t k\chi(X_s) ds} \right]^{1/2}.
		\end{align}
	 Applying \thref{lem22} for each $k$ satisfying $4N^2k^2 \leqslant \frac{k}{2}$ we get
		\[
		e^{-tL}(\ind{})(x) \lesssim e^{-kt + 4N^2k^2t} \leqslant e^{-\frac{kt}{2}},\qquad x\in\R^d.
		\]
		In particular, the above estimate holds for $k=\min(c,(8N^2)^{-1})$ and the proof is completed. 
	\end{proof}

	Now focus on our goal, which is estimating the quantity
	\begin{equation} \label{eq:Rav12}
		\Gamma(a)\,R_V^a(\ind{})(x)= V^a(x)\,\int_0^{\infty} e^{-tL}(\ind{})(x) \, t^{a-1}\,dt
	\end{equation}
	independently of $x\in \RR^d$. We will do this by splitting the integral in \eqref{eq:Rav12} into two parts and estimating them separately.
	
	Before stating the result we need to introduce a quantity $\rho$ which plays a crucial role in our assumptions. For $u \geqslant 1$ and $x\in \R^d$ we define
	\begin{equation} \label{eq:def_rho}
		\rho = \rho_x(u) = \sup \left\{r \geqslant 0: \tfrac{V(x)}{u} \leqslant V(y) \ \text{ for a.e. } y\in B(x,r) \right\};
	\end{equation}
	recall that $B(x,r)$ denotes the closed Euclidean ball of radius $r$ in $\R^d$.
	Consequently, $\rho_x(u)$ is the radius of the largest closed ball around $x$ in which the potential $V$ is at least $V(x)/u$ a.e. We note that $\rho_x(u)$ is a non-decreasing function of $u$ with values in $[0,\infty].$ We also set
	\begin{equation} \label{eq:def_rk}
		r_k = r_k(x) = \rho_x(2^{k}) \quad \text{for } k = 0, 1, \dots.
	\end{equation}
	Our main assumption will be phrased in terms of 
	\begin{equation}
		\label{eq:def:IaV}
		I^a(V)(x) \coloneqq \int_1^{\max(1, V(x))} s^{a-1} e^{-\frac{\rho_x^2(s)}{4d}} ds \qquad \text{for a.e. } x\in \R^d.
	\end{equation}
	If $\rho_x(s) = \infty$, then we define $e^{-\frac{\rho_x^2(s)}{4d}} = 0$.
	
	First we estimate the integral in \eqref{eq:Rav12} from 0 to 1. Recall that implicit constants in $\lesssim$ and $\approx$ do not depend on $x\in \R^d$ but may depend on $a>0.$
	
	\begin{lemma}
	\thlabel{lem:linf01}
		Let $V$ be an a.e. non-negative potential and let $a>0.$ Then we have
		\[
		V(x)^{a} \int_0^1 e^{-tL}(\ind{})(x) \, t^{a-1} dt \lesssim I^a(V)(x) + 1 \qquad \text{for a.e. } x\in\R^d.
		\]
	\end{lemma}
	
	\begin{proof}
		First if $V(x) \leqslant 2$, then
		\[
		V(x)^{a} \int_0^1 e^{-tL}(\ind{})(x) \, t^{a-1} dt \lesssim 1.
		\]
		
		From now on we focus on the other case $V(x)>2.$ Define $K=K(x) = \lfloor \log_2 V(x) \rfloor$. For fixed $x\in \RR^d$ and $k=0,1,2,\ldots$  we introduce the sets
		\begin{equation} \label{eq:def_Ak}
			A_k = \left\{ y \in \R^d: \frac{V(x)}{2^k} \leqslant V(y) \right\}
		\end{equation}
		and
		\begin{equation} \label{eq:def_Omega}
			\Omega_k = \left\{ \omega \in \Omega: X_s(\omega) \in A_k  \text{ for almost all } s \in[0, t] \right\},
		\end{equation}
		where $(\Omega, \mathcal{F}, \P)$ is the underlying probability space for the $d$-dimensional Brownian motion $\{X_s\}_{s > 0}$ started at 0.

		Note that both the families $\{A_k\}$ and $\{\Omega_k\}$ are increasing in $k$. Using the Feynman--Kac formula \eqref{feynman-kac1} we write
		\begin{align} \label{lem:linf01:eq1}
			e^{-tL}&(\ind{})(x) = \nonumber \\
			&\E_x \left[ e^{-\int_0^t V(X_s) ds} \, \ind{\Omega_0}  \right] + \sum_{k=1}^K \E_x \left[ e^{-\int_0^t V(X_s) ds} \, \ind{\Omega_k \cap \Omega_{k-1}^c}  \right] + \E_x \left[ e^{-\int_0^t V(X_s) ds} \, \ind{\Omega_{K}^c}  \right] \nonumber \\
			&\leqslant e^{-tV(x)}+\sum_{k=1}^K e^{-\frac{tV(x)}{2^k}} \P \left( \Omega_k\cap \Omega_{k-1}^c \right) + \P \left( \Omega_K^c \right).
		\end{align}
		We need to estimate the probabilities in the above formula. This will be achieved with the aid of
		\begin{equation} \label{eq:Pomegac}
		\P \left( \Omega_k^c \right) \leqslant \P \left( \sup_{0 \leqslant s \leqslant t} \abs{X_s-x} \geqslant r_k \right).
		\end{equation}
	
		Before moving further we focus on justifying \eqref{eq:Pomegac}.
		To prove this inequality we will show that
		\begin{equation*} 
			\left\{ \omega \in \Omega: \sup_{0 \leqslant s \leqslant t} \abs{X_s(\omega) - x} < r_k \right\} \subseteq \Omega_k
		\end{equation*}
		up to a set of $\mathbb{P}$ measure $0$. More precisely, we will demonstrate that for $\mathbb{P}$ almost all $\omega \in \Omega$ we have the implication 
		\begin{equation} \label{eq:Xsxrimp}
		\textrm{if }	\sup_{0 \leqslant s \leqslant t} \abs{X_s(\omega) - x} < r_k\quad 
	\textrm{then also }
			X_s(\omega) \in A_k  \text{ for almost all } s \in[0, t].
		\end{equation}
		
	To this end take $\omega \in \Omega$ such that $\sup_{0 \leqslant s \leqslant t} \abs{X_s(\omega) - x} < r_k$. Using the definitions \eqref{eq:def_rho} and \eqref{eq:def_rk} of $\rho$ and $r_k$ we see that there is a set $N \subseteq \R^d$ of measure 0 such that
		\[
		\text{if } X_s(\omega) \notin N \text{ then } 	\frac{V(x)}{2^k} \leqslant V(X_s(\omega)),
		\]
		By the definition \eqref{eq:def_Ak} of $A_k$ this statement is the same as the implication
		\[
		\text{if } X_s(\omega) \notin N\text{ then }	X_s(\omega) \in A_k.
		\]
		Define $f_{\omega}(s):=X_s(\omega),$ $s\in [0,t]$, and let  $\tilde{N}(\omega) = f_{\omega}^{-1}[N] \subseteq [0, t].$ Then $s \notin \tilde{N}(\omega)$ if and only if $X_s(\omega) \notin N$. We shall now demonstrate that $\abs{\tilde{N}(\omega)} = 0$ for $\mathbb{P}$ almost all $\omega \in \Omega.$ Observe that
		\[
			\abs{\tilde{N}(\omega)} = \abs{\left\{ s \in [0, t]: X_s(\omega) \in N \right\}} = \int_0^t \ind{\{X_s(\omega) \in N\}}(s, \omega) \, ds.
		\]
		Calculating the expected value of the above expression and using Fubini's theorem give
		\begin{align*}
			\E\left[ |\tilde{N}| \right] &= \E \left[ \int_0^t \ind{\{X_s(\omega) \in N\}}(s, \omega) \, ds \right] = \int_0^t \E \left[ \ind{\{X_s(\omega) \in N\}}(s, \omega) \right] \, ds \\
			&= \int_0^t \P \left( X_s(\omega) \in N \right) \, ds = 0.
		\end{align*}
		The last equality follows from the fact that $\abs{N} = 0$ and that each of the variables $X_s$ has a continuous distribution. Since $\abs{\tilde{N}(\omega)}$ is non-negative, it has to be 0 for $\mathbb{P}$ almost all $\omega \in \Omega$. 
		
		Hence we have proved that for $\mathbb{P}$ almost all $\omega \in \Omega$ there is a set $\tilde{N}(\omega) \subseteq [0, t]$ of Lebesgue measure 0 and such that
		\[
		\textrm{if } s \notin \tilde{N}(\omega)\textrm{ then }	X_s(\omega) \in A_k. 
		\]
		This proves \eqref{eq:Xsxrimp} and in consequence \eqref{eq:Pomegac}.

		Now we come back to calculating the probabilities in \eqref{lem:linf01:eq1}. The right-hand side of inequality \eqref{eq:Pomegac} is the probability that $X_s$ exits the ball of radius $r_k$ centered at $x$. We can estimate it from above by the probability that $X_s$ exits an inscribed cube whose sides are parallel to the coordinate axes. The length of its diagonal equals $a\sqrt{d} = 2r_k$, where $a$ is the cube's side length, so we get
		\begin{equation} \label{eq:erfc}
			\begin{split}
				\P\left( \sup_{0\leqslant s \leqslant t} \abs{X_s - x} \geqslant r_k \right) &\leqslant \P\left( \sup_{0\leqslant s \leqslant t} \max_{i} \abs{X^i_s - x_i} \geqslant \frac{a}{2} \right) = \P\left( \max_{i} \sup_{0\leqslant s \leqslant t} \abs{X^i_s - x_i} \geqslant \frac{a}{2} \right) \\
				&\leqslant d \cdot \P\left( \sup_{0\leqslant s \leqslant t} \abs{X^1_s - x_1} \geqslant \frac{a}{2} \right) \\
				&\leqslant d \cdot \P\left( \sup_{0\leqslant s \leqslant t} (X^1_s - x_1) \geqslant \frac{a}{2} \right) + d \cdot \P\left( \inf_{0\leqslant s \leqslant t} (X^1_s - x_1) \leqslant -\frac{a}{2} \right)\\
				&= 2d \cdot \P\left( \sup_{0\leqslant s \leqslant t} (X^1_s - x_1) \geqslant \frac{a}{2} \right) = 4d \cdot \P\left( (X_t^1 - x_1) \geqslant \frac{a}{2} \right) \\
				&\leqslant 4d \erfc \left( \frac{r_k}{\sqrt{2td}} \right) \leqslant 4d e^{-\frac{r_k^2}{2td}}.
			\end{split}
		\end{equation}
The last  equality in \eqref{eq:erfc} follows from the reflection principle for Brownian motion, while the last inequality is a well-known bound for the complementary error function $\erfc,$ see e.g.\ \cite[eq.\ (7.8.3)]{nist}.
		
		Consequently, 
		\begin{equation} \label{eq:Omkcest}
			\P \left( \Omega_k^c \right) \leqslant 4d e^{-\frac{r_k^2}{2td}}
		\end{equation}
		and coming back to \eqref{lem:linf01:eq1} for $0<t<1$ we get
		\begin{equation} \label{lem:linf01:eq2}
			\begin{split}e^{-tL}(\ind{})(x) &\lesssim e^{-tV(x)}+\sum_{k=1}^K e^{-\frac{tV(x)}{2^k}} e^{-\frac{r_{k-1}^2}{2td}} + e^{-\frac{r_{K}^2}{2td}}\\
				&\leqslant e^{-tV(x)}+\sum_{k=1}^K e^{-\frac{tV(x)}{2^k}} e^{-\frac{r_{k-1}^2}{2d}} + e^{-\frac{r_{K}^2}{2d}}
				\end{split}
		\end{equation}

		Integrating and multiplying this inequality by $V(x)^a$ gives
		\begin{equation} \label{lem:linf01:eq3}
			V(x)^{a} \int_0^1 e^{-tL}(\ind{})(x) \, t^{a-1} dt \lesssim 1+\sum_{k=1}^K 2^{ka} e^{-\frac{r_{k-1}^2}{2d}} + V(x)^{a} e^{-\frac{r_K^2}{2d}}.
		\end{equation}
	Then, for $k \geqslant 2$ we estimate each of the terms in the sum by an integral recalling that $r_k(x)=\rho_x(2^k)$ and using the fact that $\rho_x(u)$ is a non-decreasing function of $u$ 
	\begin{equation} \label{lem:linf01:eq4}
		2^{ka} e^{-\frac{r_{k-1}^2}{2d}} \leqslant \int_{k-2}^{k-1} 2^{(u+2)a} e^{-\frac{\rho_x^2(2^u)}{2d}} \, du.
	\end{equation}
	The last term in \eqref{lem:linf01:eq3} is estimated in a similar manner using additionally the fact that $V(x)^a \leqslant \int_{K-1}^K 2^{(u+2)a} \, du$. This yields
	\begin{equation} \label{lem:linf01:eq5}
		V(x)^{a} e^{-\frac{r_K^2}{2d}} \leqslant \int_{K-1}^{K} 2^{(u+2)a} e^{-\frac{\rho_x^2(2^u)}{2d}} \, du.
	\end{equation}
	We estimate the first term of the sum in \eqref{lem:linf01:eq3} by a constant and plug this, \eqref{lem:linf01:eq4} and \eqref{lem:linf01:eq5} into \eqref{lem:linf01:eq3}, which results in
	\begin{equation} \label{lem:linf01:eq5.5}
		\begin{aligned}
			1+\sum_{k=1}^K 2^{ka} e^{-\frac{r_{k-1}^2}{2d}} + V(x)^{a} e^{-\frac{r_K^2}{2d}} &\lesssim 1+\int_0^{K} 2^{ua} e^{-\frac{\rho_x^2(2^{u})}{2d}} du \\
			&\leqslant 1+\int_0^{\log_2 V(x)} 2^{ua} e^{-\frac{\rho_x^2(2^{u})}{2d}} du.
		\end{aligned}
	\end{equation}
	Finally we substitute $s = 2^u$ to get
	\begin{equation} \label{lem:linf01:eq6}
		1+\int_0^{\log_2 V(x)} 2^{ua} e^{-\frac{\rho_x^2(2^{u})}{2d}} du \approx 1+\int_1^{V(x)} s^{a-1} e^{-\frac{\rho_x^2(s)}{2d}} ds \leqslant 1+I^a(V)(x).
	\end{equation}
	\end{proof}
	In the next lemma we estimate the second part of the integral from \eqref{eq:Rav12}.

	\begin{lemma}
	\thlabel{lem:linf1inf}
		Let $V$ be an a.e. non-negative potential and suppose that, for some $\delta>0,$ the semigroup $e^{-tL}$ satisfies \eqref{eq:ed}. Take $a>0.$
		Then we have
		\begin{equation} \label{lem:linf1inf:eq1}
			V(x)^{a} \int_1^\infty e^{-tL}(\ind{})(x) \, t^{a-1} dt \lesssim I^a(V)(x)+1,\qquad x\in \R^d.
		\end{equation}
	\end{lemma}
		
	\begin{proof}
	Using the semigroup property and the positivity-preserving property of $\{e^{-tL}\}_{t>0}$ for $t \geqslant 1$ we obtain
		\begin{equation} \label{lem:linf1inf:eq2}
			\begin{split}
			e^{-tL}(\ind{})(x) &=	e^{-(t/2)L}[e^{-(t/2)L}(\ind{})](x)\leqslant \norm{e^{-(t/2)L}(\ind{})}_{\infty} e^{-(t/2)L}(\ind{})(x) \\
			&\leqslant C e^{-\delta t/2} e^{-(1/2)L}(\ind{})(x),
			\end{split}
		\end{equation}
	where the last two inequalities follow from   \eqref{eq:ed}  and \eqref{feynman-kac}.
		Plugging this into \eqref{lem:linf1inf:eq1} we get
		\begin{equation} \label{lem:linf1inf:eq3}
			V(x)^{a} \int_1^\infty e^{-tL}(\ind{})(x) \, t^{a-1} dt \lesssim V(x)^a e^{-L/2}(\ind{})(x).
		\end{equation}
		Now we are left with proving that $V^a(x) e^{-L/2}(\ind{})(x) \lesssim I^a(V)(x)+1.$ If $V(x) \leqslant 2$, then this is true. Assume that $V(x) > 2$ and let $K(x) = \lfloor \log_2 V(x) \rfloor$. Recall that by \eqref{lem:linf01:eq2} we have
		\[
		e^{-L/2}(\ind{})(x) \lesssim e^{-\frac{V(x)}{2}}+\sum_{k=1}^K e^{-\frac{V(x)}{2^{k+1}}} e^{-\frac{r_{k-1}^2}{2d}} + e^{-\frac{r_{K}^2}{2d}}.
		\]
		Since $V(x)^a e^{-\frac{V(x)}{2^{k+1}}}\le \left( \frac{2^{k+1} a}{e} \right)^a$, repeating calculations as in \eqref{lem:linf01:eq3}--\eqref{lem:linf01:eq6} we get
		\begin{equation} \label{lem:linf1inf:eq4}
			\begin{aligned}
				V(x)^a e^{-L/2}(\ind{})(x) &\lesssim 1+\sum_{k=1}^K 2^{ka} e^{-\frac{r_{k-1}^2}{2d}} + V(x)^a e^{-\frac{r_{K}^2}{2d}} \lesssim 1+I^a(V)(x).
			\end{aligned}
		\end{equation}
	In view of \eqref{lem:linf1inf:eq3} this completes the proof of the lemma. 
	\end{proof}

	Together, \thref{lem:linf01} and \thref{lem:linf1inf} lead to the following conclusion.
	
	\begin{theorem}
		\thlabel{th:linf}
		Let $V\in L^{\infty}_{\rm loc}$ be an a.e.\ non-negative potential. Suppose that the semigroup $e^{-tL}$ has exponential decay of order $\delta > 0$ (see \eqref{eq:ed}). If
		\begin{equation} \label{th:linf:eq1}
			I^a(V)\lesssim_g 1
		\end{equation}
		for some $a > 0$, then the operator $R_V^a$ is bounded on $L^\infty$.
	\end{theorem}
	\begin{proof}
		We need to estimate the quantity
		\begin{equation} \label{th:linf:eq2}
			V^a(x) \int_0^{\infty} e^{-tL}(\ind{})(x) \, t^{a-1}\,dt
		\end{equation}
		independently of $x$. Take $N > 0$ such that $I^a(V)(x) \lesssim  1$ for almost all $\abs{x} > N$. Then by \thref{lem:linf01} and \thref{lem:linf1inf} the expression $\eqref{th:linf:eq2}$ is uniformly bounded for a.e. $\abs{x} > N$. If on the other hand $\abs{x} \leqslant N$, then, since $V \in L_{\rm loc}^\infty$ and the semigroup satisfies \eqref{eq:ed}, the expression \eqref{th:linf:eq2} is uniformly bounded $x$-a.e.
	\end{proof}

	As an application of this theorem, we prove that $R_V^a$ is bounded on $L^\infty(\R^d)$ if $V$ is of the order of a power function or an exponential function. The corollary below is a restatement of one of our main results --- \thref{thm:exLinfint}.

	\begin{cor}
		\thlabel{cor:linf}
		Let  $V\colon \R^d\to [0,\infty)$ be a function in $L_{\rm loc}^\infty$. Then in all the three cases
		\begin{enumerate}
			\item $V(x)\approx 1$ globally
			\item For some $\alpha>0$ we have $V(x)\approx |x|^{\alpha}$ globally
			\item For some $\beta>1$ we have $V(x)\approx \beta^{|x|}$ globally
		\end{enumerate}
		each of the Riesz transforms $R_V^a,$ $a>0$, is bounded on $L^{\infty}(\R^d).$ 
	\end{cor}

	\begin{remark}
		More generally, the theorem also holds if in (2) and (3) we take an arbitrary norm on $\R^d$ instead of the Euclidean norm $\abs{\cdot}$. The proof is the same mutatis mutandis.
	\end{remark}

	\begin{proof}
			In the proof implicit constants in $\lesssim,$ $\gtrsim,$ and $\approx$ do not depend on $x\in \R^d$ but may depend on $a>0,$ $\alpha>0$ or $\beta>1.$
		
		Clearly in all three cases the assumptions of \thref{lem21} are satisfied, so the semigroup satisfies \eqref{eq:ed} and we only need to check that \eqref{th:linf:eq1} holds.
		
		In the first case $V(x)$ is bounded for almost all sufficiently large values of $\abs{x}$ and so is $I^a(V)(x)$ for all $a > 0$.
		
			In the second case we need to estimate from below $\rho_x(s)$ appearing in $I^a(V)$. We shall prove that $\rho_x(s)\geqslant |x|/2$ provided $s$ and $|x|$ are large enough. Let $N$ be such that for some $0<m<M$ it holds 
		\begin{equation}
			\label{eq:mxaVM}
			m\abs{x}^\alpha < V(x) < M\abs{x}^\alpha\qquad \textrm{for a.e.}\qquad \abs{x} \geqslant N.
		\end{equation}
	Take $|x|\geqslant 2N$ and assume that $|x-y|\leqslant |x|/2.$ Then $2|x|\geqslant|y|\geqslant |x|/2\geqslant N$ so that \eqref{eq:mxaVM} holds with $y$ in place of $x.$ Consequently, $V(x)\approx V(y)$ for such $x$ and $y$ so that for $s$ larger than some threshold depending only on $N$, $m$ and $M$ it holds $V(y)\geqslant V(x)/s.$  This means that for a.e. $|x|\geqslant 2N$ and uniformly large enough $s \geqslant 1$ we have $\rho_x(s)\geqslant |x|/2$. Thus, for any $a > 0$ we obtain
		\begin{equation}
		\label{eq:xaIbound}
			I^a(V)(x) \lesssim_g  1+|x|^{a\alpha}e^{-\frac{|x|^2}{16d}} \lesssim_g 1.
		\end{equation}
	 as desired.

		Finally we handle the third case.  We shall prove that $\rho_x(s)\geqslant \frac12 \min  \left( \abs{x}, \log_{\beta} s \right)$ provided $s$ and $|x|$ are large enough. Let $N>0$ be such that for some $0<m\leqslant 1 \leqslant M$ we have 
		\begin{equation}
			\label{eq:mxexpVM}
		m\beta^\abs{x} < V(x) < M\beta^{\abs{x}}\qquad \textrm{for a.e.}\qquad \abs{x} \geqslant N.
		\end{equation}	
	Take $|x|\geqslant 2N,$ $s>4,$ and assume that $|x-y|\leqslant \frac12 \min  \left( \abs{x}, \log_{\beta} s \right).$  Then, similarly to the previous paragraph, $|x|\approx |y|\geqslant N$ and \eqref{eq:mxexpVM} also holds with $y$ in place of $x.$ Therefore, for such $x$ and $y$ we have $\beta^{|y|-|x|}\approx V(y)/V(x).$ In particular $|y|-|x|-\gamma \leqslant\log_\beta V(y)-\log_\beta V(x),$ for some $\gamma>0$ independent of $x$ and $y.$ Hence, we reach
	\begin{equation}
		\label{eq:malVV}
	-\frac 12 \min  \left( \abs{x}, \log_{\beta} s \right)-\gamma \leqslant  \log_\beta V(y)-\log_\beta V(x).
	\end{equation}
	Taking $s$ large enough we see that $-\frac12 \log_{\beta} s - \gamma \geqslant -\log_\beta s$ and coming back to \eqref{eq:malVV} we obtain  
	$	V(x)/s\leqslant  V(y).$ In conclusion, we proved that $\rho_x(s)\geqslant \frac12 \min  \left( \abs{x}, \log_{\beta} s \right)$ for a.e. $|x|\geqslant 2N$ when $s$ is large enough (independently of $x$). Now, using \eqref{eq:mxexpVM}  we obtain the uniform in $|x|\geqslant 2N$ bound
		\begin{equation}
				\label{eq:bxIbound}
			I^a(V)(x) \lesssim_g 1+\int_1^{\beta^\abs{x}} s^{a-1} e^{-\frac{(\log_\beta s)^2}{16d}} ds + \int_{\beta^\abs{x}}^{M\beta^\abs{x}} s^{a-1} e^{-\frac{\abs{x}^2}{16d}} ds \lesssim_g 1,
		\end{equation}
		This completes the treatment of the third case and also the proof of \thref{cor:linf}.
	\end{proof}

	\section{$L^1$ boundedness for classes of potentials}
	\label{sec:l1}
	
	In this section we estimate the $L^1$ norm of the operator $R_V^a$ for $a>0$ and various non-negative potentials $V \in L_{\rm loc}^\infty$. Recall that the assumption $V \in L^{\infty}_{\rm loc}$ guarantees the validity of the Feynman--Kac formula \eqref{feynman-kac}. 
	
	The idea is to estimate the $L^\infty$ norm of the adjoint operator which formally is 
	\begin{equation*} 
		(L^{-a}V^a) f=\frac{1}{\Gamma(a)}\int_0^{\infty} e^{-tL}(V^a f)\,t^{a-1}\,dt.
	\end{equation*}
	Using the positivity-preserving property of $e^{-tL}$ the task naturally reduces to estimating the $L^{\infty}$ norm of the function  
	\begin{equation} \label{eq:adRav}
		\Gamma(a)L^{-a}(V^a)(x) \coloneqq \int_0^{\infty} e^{-tL}(V^a)(x)\,t^{a-1}\,dt.
	\end{equation}
	Since $V$ may be unbounded, the expression $e^{-tL}(V^a)(x)$ may be infinite for some $x$ in which case the $x$-measurability of the integral \eqref{eq:adRav} is not clear. To remedy the situation we formally define 
	\begin{equation} \label{eq:adRavfor}
		\Gamma(a)L^{-a}(V^a)(x) \coloneqq \lim_{N\to \infty}\int_0^{\infty} e^{-tL}(V^a\ind{|V|<N})(x)\,t^{a-1}e^{-t/N}\,dt.
	\end{equation} 
	Note that each of the integrals in \eqref{eq:adRavfor} is finite and measurable by \thref{lem:contfk}, hence the limit gives a measurable function by the monotone convergence theorem. A short duality argument shows that if $L^{-a}(V^a)\in L^{\infty},$ then indeed $R_V^a$ is bounded on $L^1$ with $\|R_V^a\|_{1}\leqslant \|L^{-a}(V^a)\|_{\infty}.$
	
	Throughout this section we estimate the $L^{\infty}$ norm of $L^{-a}(V^a)$ in the  form \eqref{eq:adRav}. This is allowed since by the assumptions which we will impose on $V$ both $e^{-tL}(V^a)(x)$ and the integral \eqref{eq:adRav} will turn out to be finite $x$-a.e.. This permits us to
	 take $N=\infty$ in \eqref{eq:adRavfor}.

	In what follows for $x\in \RR^d$ and $u \geqslant 1$ we let
	\[
		\sigma = \sigma_x(u) = \sup \left\{r \geqslant 0: V(y) \leqslant u V(x) \ \text{ for a.e. } y\in B(x,r) \right\}.
	\]
	Consequently, $\sigma_x(u)$ is the radius of the largest closed ball around $x$ in which the potential $V$ is at most $uV(x)$ a.e. We remark that $\sigma_x(u)$ is a non-decreasing function of $u$ with values in $[0,\infty]$. Using the quantity $\sigma_x(u)$ we define
	\begin{equation} \label{eq:Jdef}
		J^a(V)(x) \coloneqq \min(1,V(x)^a)\int_1^{\infty}s^{a-1} e^{-\sigma^2_x(s)/8}\,ds, \qquad \text{for a.e. } x\in \R^d.
	\end{equation}
	If $V\in L^{\infty}$ and $uV(x)\geqslant \|V\|_{\infty}$, then $V(y)\leqslant uV(x)$ for a.e. $y\in B(x,r)$ with arbitrarily large $r>0.$ In this case $\sigma_x(u)=\infty$ and by convention $e^{-\sigma^2_x(u)/8}=0.$ This is the case for instance if $V\in L^{\infty}$ is of constant order for large $x$.
	
	We begin with estimating the integral \eqref{eq:adRav} from 0 to 1. Recall that implicit constants in $\lesssim$ and $\approx$ are allowed to depend on $d$ and $a>0.$

	\begin{pro}
		\thlabel{lem:l101}
		Let  $V\in L^{\infty}_{\rm loc}$ be an a.e.\ non-negative potential and take $a>0$. Then the inequality
		\begin{equation} \label{eq:l101}
			\int_0^1 e^{-tL}(V^{a})(x) \, t^{a-1} dt \lesssim (J^a(V)(x)+1)(I^a(V)(x)+1)
		\end{equation}
		holds for a.e.\ $x\in \R^d$ that satisfies $V(x)\neq 0.$
		Moreover, if $V$ is an a.e.\ non-negative potential which satisfies the growth estimate $V(x)\lesssim \exp(|x|^2/(4a))$ for a.e. $x \in \R^d$, then
		\begin{equation} \label{eq:l101'}
			\int_0^1 e^{-tL}(V^{a})(x) \, t^{a-1} dt\lesssim \exp(|x|^2),\qquad x\in \R^d.
		\end{equation}
	\end{pro} 
	
	\begin{proof}

	{\bf Proof of \eqref{eq:l101}}. Here we consider $x\in\R^d$ such that  $V(x)\neq 0.$ 
			
		Recall that
		\[
			A_k = \left\{ y \in \R^d: \frac{V(x)}{2^k} \leqslant V(y)  \right\}
		\]
		and
		\[
			\Omega_k = \left\{ \omega \in \Omega: X_s(\omega) \in A_k  \text{ for almost all } s \in[0, t] \right\}.
		\]
		Here we shall also need 
		\[
			B_j=\left\{y\in \R^d\colon \, 2^{j} V(x)<V(y) \leqslant 2^{j+1} V(x) \right\}
		\]
		and
		\[
			\Psi_j = \Psi_j^t \coloneqq  \left\{ \omega \in \Omega: X_t(\omega) \in B_j \right\}.
		\]
		Note that if $V(x)\neq 0$ then the sets $\{B_j\}_{j\in \Z}$ are pairwise disjoint and 
		\begin{equation} \label{lem:l101:eq0}
			\begin{split}
				e^{-tL}(V^{a})(x)
				&=e^{-tL}\left(\sum_{j\le0} \ind{B_j} V^a\right)(x) + e^{-tL} \left(\sum_{j>0}\ind{B_j}V^a \right)(x)\\
				&\lesssim V(x)^{a} e^{-tL}(\ind{})(x) + \sum_{j>0}V(x)^{a}2^{ja} e^{-tL}(\ind{B_j})(x).
			\end{split}
		\end{equation}
			
		We shall prove that the estimates
		\begin{equation} \label{lem:l101:eq:es1}
			\int_0^1 e^{-tL}(V^{a})(x) t^{a-1} dt \lesssim (I^a(V)(x)+1) \left(\int_1^{\infty}s^{a-1} e^{-\sigma^2_x(s)/8}\,ds+1\right)
		\end{equation}
		and
		\begin{equation} \label{lem:l101:eq:es2}
		\int_0^1 e^{-tL}(V^{a})(x) t^{a-1} dt \lesssim I^a(V)(x)+1 + V(x)^a \left(\int_1^{\infty}s^{a-1} e^{-\sigma^2_x(s)/8}\,ds \right)
		\end{equation}
	hold for $x$ such that $V(x)\neq 0.$
	The inequalities \eqref{lem:l101:eq:es1} and \eqref{lem:l101:eq:es2} imply \eqref{eq:l101}.
			
		We prove \eqref{lem:l101:eq:es1} first. 
		Let $K = \max(1,\lfloor \log_2 V(x) \rfloor)$ and for $k=1,\ldots,K$ and $j\in \Z$ denote 
		\[
			r_k = \rho_x(2^{k}),\qquad s_j = \sigma_x(2^{j}).
		\]
		Estimating the second term in \eqref{lem:l101:eq0} we use the Feynman--Kac formula \eqref{feynman-kac} with $f=V^{a}\ind{B_j}$ to write
		\begin{equation} \label{lem:l101:eq1}
			\sum_{j> 0}e^{-tL}(V^{a}\ind{B_j})(x)\lesssim V^a(x)\sum_{j> 0}2^{ja}e^{-tL}(\ind{B_j})(x).
		\end{equation}
Using again \eqref{feynman-kac}, proceeding as in the proof of \thref{lem:linf01} and applying \eqref{eq:Omkcest} we  obtain 
		\begin{align*}
				e^{-tL}(\ind{B_j})(x) &\leqslant e^{-tV(x)}\P\left(\Psi_j\right)+\sum_{k=1}^K e^{-\frac{tV(x)}{2^k}}\P \left( \Omega_{k-1}^c \cap \Psi_j\right)+\P \left( \Omega_K^c \cap \Psi_j\right)\\
				&\leqslant \P ( \Psi_j)^{1/2}\left( e^{-tV(x)}+\sum_{k=1}^K e^{-\frac{tV(x)}{2^k}}\left[\P \left( \Omega_{k-1}^c\right)\right]^{1/2}+\left[\P \left( \Omega_{K}^c\right)\right]^{1/2}\right)\\
				&\lesssim \P ( \Psi_j)^{1/2}\left( e^{-tV(x)}+\sum_{k=1}^K e^{-\frac{tV(x)}{2^k}}e^{-\frac{r_{k-1}^2}{4td}}+e^{-\frac{r_{K}^2}{4td}}\right)
		\end{align*}
		Further, we have $\Psi_j\subseteq \{\omega \in \Omega \colon X_t(\omega) \not \in B(x,s_j)\}$ up to a set of $\mathbb{P}$ measure $0$. Indeed, a.e.\ $y\in B(x,s_j)$ satisfies $V(y)\le 2^jV(x),$ hence it lies outside $B_j.$ Here we also use the fact that $X_t$ has a continuous distribution. Thus we reach
		\begin{equation} \label{lem:l101:eqPpsij}
			\begin{aligned}
				\P(\Psi_j) &\leqslant \P(|X_t-x|\geqslant s_j) = \frac{1}{(2\pi t)^{d/2}} \int_{\abs{y} \geqslant s_j} e^{-\frac{\abs{y}^2}{2t}} dy \\
				&\leqslant \frac{e^{-s_j^2/(4t)}}{(2\pi t)^{d/2}}  \int_{\abs{y} \geqslant s_j} e^{-\frac{\abs{y}^2}{4t}} dy \lesssim e^{-s_j^2/(4t)}
			\end{aligned} 
		\end{equation}
		so that
		\[
		e^{-tL}(\ind{B_j})(x)\lesssim e^{-s_j^2/(8t)} \left( e^{-tV(x)}+\sum_{k=1}^K e^{-\frac{tV(x)}{2^k}}e^{-\frac{r_{k-1}^2}{4td}}+e^{-\frac{r_{K}^2}{4td}}\right).
		\]
		Putting the above bound in \eqref{lem:l101:eq0} and replacing the sum over $j$ with an integral as in \eqref{lem:linf01:eq5} and \eqref{lem:linf01:eq5.5} we reach
		\begin{align*}
			&\sum_{j>0}V(x)^{a}2^{ja}e^{-tL}(\ind{B_j})(x)\lesssim V(x)^{a}  \left( e^{-tV(x)}+\sum_{k=1}^K e^{-\frac{tV(x)}{2^k}}e^{-\frac{r_{k-1}^2}{4td}}+e^{-\frac{r_{K}^2}{4td}}\right)\sum_{j>0} 2^{ja} e^{-s_j^2/(8t)}\\
			&\lesssim V(x)^a \left( e^{-tV(x)}+\sum_{k=1}^K e^{-\frac{tV(x)}{2^k}} e^{-\frac{r_{k-1}^2}{4td}} + e^{-\frac{r_{K}^2}{4td}}\right) \int_1^{\infty}s^{a-1} e^{-\sigma^2_x(s)/(8t)}\,ds.
		\end{align*}
		The first term on the right-hand side of \eqref{lem:l101:eq0} was already estimated in the proof of \thref{lem:linf01} by
		\[
			V(x)^{a} e^{-tL}(\ind{})(x)\leqslant V(x)^a \left( e^{-tV(x)}+\sum_{k=1}^K e^{-\frac{tV(x)}{2^k}} e^{-\frac{r_{k-1}^2}{2td}} + e^{-\frac{r_{K}^2}{2td}}\right),
		\]
		see \eqref{lem:linf01:eq2}. Hence, coming back to \eqref{lem:l101:eq0} we reach 
		\[
			e^{-tL}(V^{a})(x)\lesssim V(x)^a \left(\int_1^{\infty}s^{a-1} e^{-\sigma^2_x(s)/8}\,ds+1\right)\left(e^{-tV(x)}+\sum_{k=1}^K e^{-\frac{tV(x)}{2^k}} e^{-\frac{r_{k-1}^2}{4td}} + e^{-\frac{r_{K}^2}{4td}}\right)
		\]
		We use the above inequality to estimate $\int_0^1 	e^{-tL}(V^{a})(x)\,t^{a-1}\,dt.$ From this point on the proof is a  repetition of the argument in \eqref{lem:linf01:eq3}--\eqref{lem:linf01:eq6} that leads to \eqref{lem:l101:eq:es1}.
			
		Now we pass to the proof of \eqref{lem:l101:eq:es2}. This time we merely estimate $e^{-tL}(\ind{B_j})(x)$ by $\P(\Psi_j).$ In view of \eqref{lem:l101:eq0} and \eqref{lem:l101:eqPpsij} proceeding as in the proof of \eqref{lem:l101:eq:es1} we thus obtain
		\begin{align*}
			&e^{-tL}(V^a)(x)\lesssim V(x)^a	\left(e^{-tV(x)}+\sum_{k=1}^K e^{-\frac{tV(x)}{2^k}} e^{-\frac{r_{k-1}^2}{2td}} + e^{-\frac{r_{K}^2}{2td}} \right) + V(x)^a\sum_{j>0}2^{ja} e^{-s_j^2/(4t)} \\
			&\lesssim V(x)^a \left(e^{-tV(x)}+\sum_{k=1}^K e^{-\frac{tV(x)}{2^k}} e^{-\frac{r_{k-1}^2}{2td}} + e^{-\frac{r_{K}^2}{2td}}\right)+V(x)^a \int_1^{\infty}s^{a-1} e^{-\sigma^2_x(s)/8}\,ds.
		\end{align*}
		Once again we integrate the above expression by repeating the argument in \eqref{lem:linf01:eq3}--\eqref{lem:linf01:eq6} and obtain \eqref{lem:l101:eq:es2}.

		{\bf Proof of \eqref{eq:l101'}} The growth assumption on $V$ implies that
		\[
			\E_x[V(X_t)^a]\lesssim (2\pi t)^{-d/2}\int_{\R^d} e^{-|y-x|^2/(2t)}e^{|y|^2/4}\,dy. 
		\]
		Then, a short calculation leads to
		\begin{equation} \label{lem:l101:eq:ExVest}
			\E_x[V(X_t)^a] \lesssim  \exp(|x|^2),\qquad t<1.
		\end{equation}
		Thus, using the Feynman--Kac formula \eqref{feynman-kac} we estimate 
		\[
			e^{-tL}(V^a)(x) \leqslant \E_x[V(X_t)^a]\lesssim \exp(|x|^2),
		\]
		so that
		\[
			\int_0^1 e^{-tL}(V^{a})(x) \, t^{a-1} dt\lesssim \exp(|x|^2).
		\]
		
		This completes the proof of \thref{lem:l101}.
			
		\end{proof}

Now we pass to the integral \eqref{eq:adRav} restricted to the range $[1,\infty)$. We shall prove several results with varying assumptions on the potential $V.$
For this reason the treatment here is significantly more complicated than in Section \ref{sec:linf}.

We start with a counterpart of \thref{lem:l101}. To this end we need yet another quantity 
	\begin{equation}
		\label{eq:tJdef}
		K^{a}_{c}(V)(x) \coloneqq \min(1,V(x)^a)\int_1^{\infty} e^{-c\sigma_x(s)}s^{a-1}\,ds, \qquad \text{for a.e. } x\in \R^d,
	\end{equation}
	where $a,c>0.$ Note that this is essentially larger than $J^a(V)(x)$ defined by \eqref{eq:Jdef} and used in \thref{lem:l101}. Indeed, observe that for each $c>0$ there is a constant $M$ independent of $x$ and $s$ such that $\frac{\sigma_x^2(s)}{8} \geqslant c\sigma_x(s) - M$ for all $s \geqslant 1$ and $x \in \R^d$, which means that $e^{-\sigma_x^2(s)/8} \leqslant e^{M} e^{-c\sigma_x(s)}$ and in turn
	\begin{equation}
		\label{eq:comKJ}
		J^a(V)(x)\lesssim K^a_c(V)(x).
	\end{equation}

	\begin{pro}
	\thlabel{lem:l101lt}
		Let  $V$ be an a.e. non-negative potential. Assume that the semigroup $e^{-tL}$ satisfies \eqref{eq:ed} with some $\delta>0.$ Let $a>0$, take $b>a$ and define \begin{equation}
			\label{eq:cdef}
			c=\min\left(\frac{b-a}{8b},\frac{\delta a}{4b}\right).\end{equation} Then 
		\begin{equation}
			\label{eq:l101lt}
			\int_1^{\infty} e^{-tL}(V^{a})(x) \, t^{a-1} dt \lesssim  (	K^{a}_{c}(V)(x) + 1) (I^b(V)(x)+1)
		\end{equation}
		uniformly in every $x$ such that $V(x)\neq 0.$ 
		
		Moreover, if $V$ is of exponential growth $\eta$, i.e.
		\begin{equation}
			\label{eq:expgro}
			V(x)\lesssim e^{\eta|x|},
		\end{equation}
		with $\eta<\sqrt{\delta}/(\sqrt{2d}a),$ then
		\begin{equation}
			\label{eq:l101lt'}
			\int_1^{\infty} e^{-tL}(V^{a})(x) \, t^{a-1} dt\lesssim \exp(\sqrt{d}a\eta|x|),\qquad x\in \R^d.
		\end{equation}
\end{pro}
\begin{remark}
		The implicit constants in \eqref{eq:l101lt}, \eqref{eq:l101lt'} possibly depend on $a,b,\delta,\eta.$
\end{remark}
	
	\begin{proof}	
		{\bf Proof of \eqref{eq:l101lt}}.
		Using the splitting into the sets $B_j$ as in \eqref{lem:l101:eq0} and  the Feynman--Kac formula \eqref{feynman-kac} we obtain
		\begin{align*}
			e^{-tL}(V^{a})(x) &\lesssim  V(x)^{a} e^{-tL}(\ind{})(x)+\sum_{j>0}V(x)^{a}2^{ja}e^{-tL}(\ind{B_j})(x)\\
			&\lesssim V(x)^{a} e^{-tL}(\ind{})(x)+\sum_{j>0}V(x)^{a}2^{ja}\E_x[ e^{-\int_0^t V(X_s) ds}\ind{\Psi_j}]
		\end{align*}
		By \thref{lem:linf1inf} we have
		\[
			\int_1^{\infty} V(x)^{a} e^{-tL}(\ind{})(x)\,t^{a-1}\,dt\lesssim I^a(V)(x)+1 \lesssim I^b(V)(x)+1.
		\]
		Hence, we only focus on the integral over the second term, namely 
		$\int_1^{\infty} S_x(t)\,t^{a-1}\,dt$
		with
		\begin{equation}
			\label{lem:l101lt:eq:st}
			S_x(t) \coloneqq \sum_{j>0}V(x)^{a}2^{ja}\E_x[ e^{-\int_0^t V(X_s) ds}\ind{\Psi_j}].
		\end{equation}
		
		Let $p=b/a$ and let $q$ be its conjugate exponent. Then H\"older's inequality gives
		\begin{equation}
			\label{lem:l101lt:eq:stest0}
			\begin{split}
				S_x(t)&\leqslant \sum_{j>0}V(x)^{a}2^{ja}\left(\E_x[ e^{-p\int_0^t V(X_s) ds}]\right)^{1/p}\left(\E_x[\ind{\Psi_j}]\right)^{1/q}\\
				&\lesssim \sum_{j>0}V(x)^{a}2^{ja}\left(e^{-tL}(\ind{})(x)\right)^{1/p}\P(\Psi_j)^{1/q}.
			\end{split}
		\end{equation}
		Using \eqref{lem:l101lt:eq:stest0} we shall prove that
		\begin{equation}
			\label{lem:l101lt:eq:es1}
			\int_1^{\infty} S_x(t)\,t^{a-1}\,dt \lesssim (I^b(V)(x)+1) \left(\int_1^{\infty}e^{-c\sigma_x(s)}s^{a-1}\,ds+1\right).
		\end{equation}
		and
		\begin{equation}
			\label{lem:l101lt:eq:es2}
			\int_1^{\infty} S_x(t)\,t^{a-1}\,dt \lesssim V(x)^a\left(\int_1^{\infty} e^{-c\sigma_x(s)}s^{a-1}\,ds\right).
		\end{equation}
		These two inequalities imply that
		\[
		\int_1^{\infty} S_x(t)\,t^{a-1}\,dt\lesssim (K^{a}_{c}(V)(x)+1) (I^b(V)(x)+1),
		\]
		and thus are enough to complete the proof of \eqref{eq:l101lt}.
		
		We start with \eqref{lem:l101lt:eq:es1}. Using monotonicity, the semigroup property, and \eqref{eq:ed}  we obtain that 
		\[
		e^{-tL}(\ind{})(x)=e^{-tL/2}(e^{-tL/2}(\ind{}))(x)\lesssim e^{-\delta t/2 }e^{-L/2}(\ind{})(x). 
		\]
		Hence, \eqref{lem:l101lt:eq:stest0} gives
		\begin{align*}
			S_x(t)&\leqslant e^{-\delta t/(2p) }\left(V(x)^{ap}e^{-L/2}(\ind{})(x)\right)^{1/p}\cdot \sum_{j>0}2^{ja}\P(\Psi_j)^{1/q}.
		\end{align*}
		Since $ap= b$ a repetition of the computation in \eqref{lem:linf1inf:eq4} shows that
		\begin{equation}
			\label{lem:l101lt:eq:Sx1}
			S_x(t)\lesssim (I^b(V)(x)+1)\cdot e^{-\delta t/(2p)}\cdot \sum_{j>0}2^{ja}\P(\Psi_j)^{1/q}.
		\end{equation}
		Now, using the estimate \eqref{lem:l101:eqPpsij} for $\P(\Psi_j)$ we obtain
		\begin{equation}
			\label{lem:l101lt:eq:sPj}
			\sum_{j>0}2^{ja}\P(\Psi_j)^{1/q}\lesssim \sum_{j>0}2^{ja} e^{-s_j^2/(4tq)}.
		\end{equation}
		Consider the integral
		\[
			\int_1^{\infty} e^{-\delta t /(2p)}e^{-s_j^2/(4tq)}t^{a-1}\,dt.
		\]
		We split it at $t = s_j$ and estimate each part separately:
		\begin{align*}
			\int_1^{\infty} e^{-\delta t /(2p)} &e^{-s_j^2/(4tq)}t^{a-1}\,dt \leqslant \int_1^{s_j} e^{-s_j^2/(4tq)}t^{a-1}\,dt +\int_{s_j}^{\infty} e^{-\delta t /(2p)}t^{a-1}\,dt \\
			&\lesssim e^{-s_j/(8q)} + e^{-\delta s_j/(4p)} \lesssim e^{-c s_j}.
		\end{align*}
Recall that $c=\min((b-a)/(8b),\delta a/(4b)).$ Formally, the splitting above only works when $s_j\geqslant 1,$ however, the estimate
	$$\int_1^{\infty} e^{-\delta t /(2p)} e^{-s_j^2/(4tq)}t^{a-1}\,dt\lesssim  e^{-c s_j}$$ remains true for any $s_j\geqslant 0$.
Consequently, integrating \eqref{lem:l101lt:eq:sPj} we get
		\begin{equation}
			\label{lem:l101lt:eq:intsumsPj}
			\int_1^{\infty}  e^{-\delta t/(2p)}\cdot \sum_{j>0}2^{ja}\P(\Psi_j)^{1/q}t^{a-1}\,dt \leqslant \sum_{j>0} 2^{ja} e^{-c s_j}\lesssim \int_1^{\infty}e^{-c \sigma_x(s)}s^{a-1}\,ds,
		\end{equation}
	where in the last inequality above we used the fact that $s_j=\sigma_x(2^{j}).$ Combining \eqref{lem:l101lt:eq:intsumsPj} with \eqref{lem:l101lt:eq:Sx1} gives \eqref{lem:l101lt:eq:es1}.
		
		We pass to the proof of \eqref{lem:l101lt:eq:es2}. Note that \eqref{lem:l101lt:eq:stest0} and the assumption \eqref{eq:ed} imply
		\begin{equation*}
			S_x(t)\lesssim e^{-\delta t/p} \sum_{j>0}V(x)^{a}2^{ja}\P(\Psi_j)^{1/q},
		\end{equation*}
		thus, an application of \eqref{lem:l101lt:eq:intsumsPj} produces
		\[
		\int_1^{\infty}	S_x(t)\,t^{a-1}\,dt \lesssim V(x)^a\int_1^{\infty}e^{-c\sigma_x(s)}s^{a-1}\,ds,
		\]
		and \eqref{lem:l101lt:eq:es2} is justified.
		
		{\bf Proof of \eqref{eq:l101lt'}}. Using the Feynman--Kac formula \eqref{feynman-kac1} and Cauchy--Schwarz inequality we obtain
		\begin{align*}
			e^{-tL}(V^{a})(x)&\leqslant \E_x \left[ V^{2a}(X_t) \right]^{1/2}\E_x \left[ e^{-2\int_0^t V(X_s) ds}  \right]^{1/2}\\
			&\leqslant \E_x \left[  V^{2a}(X_t) \right]^{1/2}\left(e^{-tL}(\ind{})(x)\right)^{1/2}.
		\end{align*}
		Hence, the assumptions \eqref{eq:ed} and \eqref{eq:expgro} give
		\[
			e^{-tL}(V^{a})(x) \lesssim e^{-\delta t /2} \left(\E_x   e^{2\eta a|X_t|} \right)^{1/2}.
		\]
		We claim that the proof of  \eqref{eq:l101lt'} will be completed if we show that
		\begin{equation}
			\label{lem:l101lt:eq:claim1}
			\E_x   e^{2\eta a|X_t|}\lesssim \exp(2d\eta^2 a^2 t+ 2\sqrt{d}\eta a |x|).
		\end{equation}
		Indeed, the above estimate  leads to
		\begin{align*}
			\int_1^{\infty} e^{-tL}(V^{a})(x)\,t^{a-1}\,dt\lesssim  e^{\sqrt{d}\eta a|x|}\int_1^{\infty} \exp(-\delta t /2+d\eta^2 a^2 t)\,t^{a-1}\,dt\lesssim e^{\sqrt{d}\eta a|x|}, \end{align*}
		where in the last inequality we used the assumption $\eta<\sqrt{\delta}/(\sqrt{2d}a).$
		
		It remains to justify \eqref{lem:l101lt:eq:claim1}. Since
		\begin{equation}
			\label{lem:l101lt:eq:claim1pf}
			\begin{split}
				\E_x \left[ e^{2\eta a \abs{X_t}} \right] &= \frac{1}{\left( 2\pi t \right)^{d/2}} \int_{\R^d} e^{2\eta a \abs{z}} e^{-\frac{\abs{x-z}^2}{2t}} dz\leqslant \frac{1}{\left( 2\pi t \right)^{d/2}} \int_{\R^d} e^{2\eta a \sum_{i=1}^d|z_i|} e^{-\frac{\abs{x-z}^2}{2t}} dz\\
				&=\prod_{i=1}^d \frac{1}{\sqrt{2\pi t}} \int_{\R} e^{2\eta a |z_i|} e^{-\frac{\abs{x_i-z_i}^2}{2t}} dz_i
			\end{split}
		\end{equation}
		it suffices to focus on each of the factors in the above product separately. A simple computation shows that
		\begin{align*}
			&\frac{1}{\sqrt{2\pi t}}\int_{\R} e^{2\eta a |z_i|} e^{-\frac{\abs{x_i-z_i}^2}{2t}} dz_i\leqslant e^{2\eta a |x_i|}\frac{1}{\sqrt{2\pi t}}\int_{\R} e^{2\eta a |z_i-x_i|} e^{-\frac{\abs{x_i-z_i}^2}{2t}} dz_i\\
			&= e^{2\eta a |x_i|}\frac{1}{\sqrt{2\pi t}}\int_{\R} e^{2\eta a |y|} e^{-\frac{\abs{y}^2}{2t}} dy\leqslant 2 e^{2\eta a |x_i|}\frac{1}{\sqrt{2\pi t}}\int_{\R} e^{2\eta a y} e^{-\frac{\abs{y}^2}{2t}} dy\\
			&=2e^{2\eta a |x_i|}e^{(2\eta a)^2t/2}=2e^{2\eta a |x_i|}e^{2\eta^2a^2 t}.
		\end{align*}
		Hence, coming back to \eqref{lem:l101lt:eq:claim1pf} and using the inequality $\sum_{i=1}^d|x_i|\leqslant \sqrt{d}|x|$  we obtain
		\[
			\E_x \left[ e^{2\eta a \abs{X_t}} \right]\leqslant 2^{d}e^{2d\eta^2a^2t} \prod_{i=1}^de^{2\eta a |x_i|} \lesssim \exp(2d\eta^2 a^2 t+ 2\sqrt{d}\eta a |x|),
		\]
		thus proving the claim \eqref{lem:l101lt:eq:claim1}. 
		
		The proof of \thref{lem:l101lt} is thus completed.
		
	\end{proof}

By a comparison with the Hermite semigroup we can improve \thref{lem:l101lt} in the full range $a>0$ for potentials $V$ which grow at infinity faster than $|x|^2.$

\begin{pro}
\thlabel{lem28}
Let $c,b,N$ be positive constants. Assume that $V \in L_{\rm loc}^\infty$ is an a.e. non-negative potential that satisfies   $c|x|^2\leqslant V(x)$ for a.e. $|x|\geqslant N$ and $V(x) \lesssim e^{b|x|^2}$. Denote $\mu=\frac{d^{1/3}}{5 N^2}.$ Then, for each $0 < a \leqslant \frac{\mu \tanh \frac{\mu}{2}}{4b}$ we have
	\begin{equation} \label{eq28}
		\int_1^\infty e^{-tL}(V^{a})(x) \, t^{a-1} dt \lesssim 1,\qquad x\in\R^d.
	\end{equation}
\end{pro}

\begin{proof}
	Denote by $\omega$ a $C_c^{\infty}$ function which is equal to $c|x|^2$ for $|x|\leqslant N,$ is bounded by $c|x|^2,$ and vanishes for $|x|\geqslant 2N.$ Then, for all $k\in(0,1],$ we have
	\[
	V(x)+k\omega(x)\geqslant ck|x|^2,\qquad \text{for a.e. } x\in\R^d.
	\]
Hence, using \eqref{feynman-kac1} and Cauchy--Schwarz inequality we obtain
	\begin{equation}
		\label{lem:28:eq:1}
		\begin{split}
		e^{-tL}(V^{a})(x)&=	 \E_x \left[ e^{-\int_0^t V(X_s) ds} V^a(X_t) \right]=\E_x \left[ e^{-\int_0^t (V+k\omega)(X_s) ds} V^a(X_t) \cdot e^{k\int_0^t \omega(X_s) ds} \right]
		\\&\leqslant \left(\E_x \left[ e^{-2\int_0^t (V+k\omega)(X_s) ds} V^{2a}(X_t) \right]\right)^{1/2}\cdot \left(\E_x  e^{2k\int_0^t \omega(X_s) ds} \right)^{1/2}\\
		&\leqslant \left(\E_x \left[ e^{-2ck \int_0^t |X_s|^2 ds} V^{2a}(X_t) \right]\right)^{1/2}\cdot \left(\E_x  e^{2k\int_0^t \omega(X_s) ds} \right)^{1/2}\\
		&=\left(e^{-t(-\frac{\Delta}{2} + 2ck\abs{x}^2)}(V^{2a})(x)\right)^{1/2}\cdot \left(\E_x  e^{2k\int_0^t \omega(X_s) ds} \right)^{1/2}.
		\end{split}
	\end{equation}

In what follows we denote
	$$\gamma=\gamma(c,k)=2\sqrt{ck}.$$
	Throughout the proof the implicit constants in $\lesssim$ depend on $k\in(0,1],$ thus also on $\gamma.$ Appropriate $k$ and $\gamma$ will be fixed at a later stage. From \cite[4.1.2]{thangavelu} or \cite[1.4]{stempak} we deduce that
	\[
		e^{-t(-\frac{\Delta}{2} + 2ck\abs{x}^2)}f(x)=e^{-t(-\frac{\Delta}{2} + \frac{\gamma^2}{2}\abs{x}^2)}f(x) = \left(\frac{\gamma}{2\pi} \right)^{d/2} \int_{\R^d} K_t^\gamma(x,y)f(y) \, dy,
	\]
	with
	\begin{align*}
		K_t^\gamma(x,y) &= \frac{1}{(\sinh \gamma t)^{d/2}} \exp \left( -\frac{\gamma}{2}\left( \abs{x}^2+\abs{y}^2 \right)\coth \gamma t + \frac{\gamma \innprod{x}{y}}{\sinh \gamma t} \right) \\
		&= \frac{1}{(\sinh \gamma t)^{d/2}} \exp \left( -\frac{\gamma\abs{x-y}^2}{4\tanh \frac{\gamma t}{2}} - \frac{\gamma\tanh \frac{\gamma t}{2}}{4} \abs{x+y}^2 \right).
	\end{align*}

	Using the upper bound on $V$ we estimate $e^{-t(-\frac{\Delta}{2} + \frac{\gamma^2}{2}\abs{x}^2)}(V^{2a})$ as follows
	\begin{align} \label{lem28:eq3}
		e^{-t(-\frac{\Delta}{2} + \frac{\gamma^2}{2}\abs{x}^2)}&(V^{2a})(x) \nonumber \\
		&\lesssim \frac{1}{(\sinh \gamma t)^{d/2}} \int_{\R^d} V(y)^{2a} \exp \left( -\frac{\gamma\abs{x-y}^2}{4\tanh \frac{\gamma t}{2}} - \frac{\gamma \tanh \frac{\gamma t}{2}}{4} \abs{x+y}^2 \right) dy \nonumber \\
		&\lesssim e^{-\frac{d\gamma t}{2}} \int_{\R^d} \exp \left(2ab\abs{y}^2 -\frac{\gamma\abs{x-y}^2}{4\tanh \frac{\gamma t}{2}} - \frac{\gamma\tanh \frac{\gamma t}{2}}{4} \abs{x+y}^2 \right) dy
	\end{align}
Rewriting the exponents we obtain
	\begin{align*}
		&2ab\abs{y}^2 -\frac{\gamma \abs{x-y}^2}{4\tanh \frac{\gamma t}{2}} - \frac{\gamma\tanh \frac{\gamma t}{2}}{4} \abs{x+y}^2 \\
		&= \left( 2ab - \frac{\gamma \coth \gamma t}{2} \right) \abs{ y + \frac{\gamma \csch \gamma t}{4ab - \gamma \coth \gamma t} x }^2 - \left( \frac{\gamma \coth \gamma t}{2} + \frac{\left( \gamma \csch \gamma t \right)^2}{8ab - 2\gamma \coth \gamma t} \right)\abs{x}^2.
	\end{align*}
	We see that for the integral in \eqref{lem28:eq3} to be finite the quantity $\varphi(t) \coloneqq 2ab - \frac{\gamma \coth \gamma t}{2}$ has to be negative for all $t \geqslant 1$, which is satisfied for $a \leqslant \frac{\gamma \tanh \frac{\gamma}{2}}{4b}$ since $\frac{\gamma \tanh \frac{\gamma}{2}}{4b}< \frac{\gamma \coth \gamma t}{4b}.$ For such $a$ we have $\varphi(t) \leqslant \frac{\gamma}{2}(\tanh \frac{\gamma}{2}-\coth \gamma t)$ and
	\begin{align*}
		&\int_{\R^d} \exp \left(2ab\abs{y}^2 -\frac{\gamma\abs{x-y}^2}{4\tanh \frac{\gamma t}{2}} - \frac{\gamma\tanh \frac{\gamma t}{2}}{4} \abs{x+y}^2 \right) dy \\
		&= \exp \left( - \left( \frac{\gamma \coth \gamma t}{2} + \frac{\left( \gamma \csch \gamma t \right)^2}{4\varphi(t)} \right)\abs{x}^2 \right) \int_{\R^d} e^{\varphi(t) \abs{y}^2} dy \\
		&\leqslant \exp \left( - \frac{\gamma}{2}\left( \coth \gamma t + \frac{\csch^2 \gamma t}{\tanh \frac{\gamma}{2} - \coth \gamma t} \right)\abs{x}^2 \right) \left( -\frac{\pi}{\varphi(t)} \right)^{d/2}.
	\end{align*}
	Denoting $\psi(t) := \coth \gamma t + \frac{\csch^2 \gamma t}{\tanh \frac{\gamma}{2} - \coth \gamma t}$ a calculation gives $$\psi'(t) = -\frac{\gamma \csch^2 \gamma t \ \cdot \ \left( -1 + \tanh^2 \frac{\gamma}{2} \right)}{\left( \tanh \frac{\gamma}{2} - \coth \gamma t \right)^2}.$$
Since $\psi'$ is positive the function $\psi$ is strictly increasing. Moreover it  has a zero at $t=\frac{1}{2}$ so that for $t \geqslant 1$ we have $\psi(t) \geqslant \psi(1) =\delta > 0$ and thus we can continue the previous calculation as follows
	\begin{align*}
	&\exp \left( - \frac{\gamma}{2}\left( \coth \gamma t + \frac{\csch^2 \gamma t}{\tanh \frac{\gamma}{2} - \coth \gamma t} \right)\abs{x}^2 \right) \left( -\frac{\pi}{\varphi(t)} \right)^{d/2} \\
	&\lesssim e^{-\frac{\gamma\delta\abs{x}^2}{2}}(-\varphi(t))^{-d/2}
	\end{align*}
Next we need to handle the term $(-\varphi(t))^{-d/2}$. Since 
	$a \leqslant \frac{\gamma \tanh \frac{\gamma}{2}}{4b} $ we see that $$(-\varphi(t))^{-d/2} \lesssim \left(\gamma \left( \coth \gamma t - \tanh \frac{\gamma}{2} \right)\right)^{-d/2}\lesssim 1,\qquad t\geqslant 1.$$ 
	Finally plugging the above estimates in \eqref{lem28:eq3} we get
	\begin{equation} \label{lem:28:eq:2}
		e^{-t(-\frac{\Delta}{2} + \frac{\gamma^2}{2}\abs{x}^2)}(V^{2a})(x) \lesssim e^{-\frac{d\gamma t}{2}} e^{-\frac{\gamma\delta\abs{x}^2}{2}},
	\end{equation}
uniformly in $x\in \R^d$ and $t\geqslant 1.$
	
	Next we estimate $\left(\E_x  e^{2k\int_0^t \omega(X_s) ds} \right)^{1/2}.$ Since $\omega \leqslant 4cN^2 \ind{P}$ for $P = [-2N, 2N] \times \R^{d-1}$, we can apply \thref{lem22} with $k' = 4ckN^2$, which gives
	\begin{equation} \label{lem:28:eq:3}
		\E_x  e^{2k\int_0^t \omega(X_s) ds} \lesssim e^{512 c^2 k^2 N^6 t} = e^{32 \gamma^4 N^6 t}
	\end{equation}
	 
	 Combining \eqref{lem:28:eq:2} and \eqref{lem:28:eq:3} and coming back to \eqref{lem:28:eq:1} we reach

	\begin{align*}
		\int_1^\infty e^{-tL}(V^{a})(x) \, t^{a-1} dt &\lesssim e^{-\frac{\gamma\delta\abs{x}^2}{4}} \int_1^\infty e^{-\frac{d\gamma t}{4}} e^{16 \gamma^4 N^6 t} t^{a-1} dt \lesssim 1,\qquad x\in\R^d,
	\end{align*}
provided that $\gamma< \frac{d^{1/3}}{4N^2}.$ This can be achieved by taking $k=\min(1,\mu^2/(4c)),$ since for such  $k$  we have
\[
	\gamma = 2\sqrt{ck} \leqslant \mu < \frac{d^{1/3}}{4N^2}.
\]

The proof of \thref{lem28} is thus completed.

\end{proof}

	We shall now derive $L^1$ boundedness of $R_V^a$ using  \thref{lem:l101} together with one of the \thref{lem:l101lt,lem:l11inf,lem28}.
	
	Combining  \thref{lem:l101} and \thref{lem:l101lt} we get a theorem on the $L^1$ boundedness of $R^a_V.$ Note that this theorem inherits the stronger assumptions on $V$ from \thref{lem:l101lt}. Its advantage is the allowance of large $a$ when $V(x)\lesssim e^{\eta|x|}$ with small $\eta.$ This is useful for instance when $V(x)\ag |x|^{\alpha}.$
	
	\begin{theorem}
		\thlabel{th:l1}
		Let $V$ be an a.e. non-negative potential having an exponential growth \eqref{eq:expgro} for some $\eta>0$ and such that $e^{-tL}$ has an exponential decay \eqref{eq:ed} of an order $\delta > 0.$ Let $0<a<\delta^{1/2}(2d)^{-1/2} \eta^{-1},$ take $b>a$ and let $c$ be the constant defined in \eqref{eq:cdef}. If
		\[
			K_c^a(V)(x) \lesssim_g 1 \qquad \text{and} \qquad I^b(V)(x) \lesssim_g 1,
		\]
		then $R_V^a$ is bounded on $L^1.$
	\end{theorem}
	\begin{proof}
		By duality it suffices to estimate the $L^{\infty}$ norm of 
		\begin{equation}
			\label{eq:LGsplit}
			\begin{split}
		\frac{1}{\Gamma(a)}\int_0^{\infty} e^{-tL}(V^a)t^{a-1}\,dt&=\frac{1}{\Gamma(a)}\int_0^1 e^{-tL}(V^a)t^{a-1}\,dt+\frac{1}{\Gamma(a)}\int_1^{\infty} e^{-tL}(V^a)t^{a-1}\,dt\\
		&=:L+G.
		\end{split}
	\end{equation}
		Using the bound $e^{\eta|x|}\lesssim e^{|x|^2/(4a)}$ and \eqref{eq:l101'} from \thref{lem:l101} we see that
		\[
		L(x)\lesssim C(N), 
		\]
		whenever $|x|\leqslant N.$ Then \eqref{eq:comKJ} together with \eqref{eq:l101} from \thref{lem:l101}  gives
		\[
		\|L\|_{\infty}\lesssim 1.
		\]
		The estimate
		\[
		\|G\|_{\infty}\lesssim 1
		\]
		is a straightforward consequence of our assumptions and \thref{lem:l101lt}.
		
	\end{proof}

\thref{lem:l101} and \thref{lem28} allow us to improve \thref{th:l1} for potentials that grow at least as a constant times $|x|^2$. The improvement comes from the replacement of the condition $K_c^a(V)(x)\lesssim_g 1$ by $J^a(V)(x)\lesssim 1.$ This is useful e.g.\ for potentials $V(x)=\beta^{|x|},$ $\beta>1,$ for which  $K_c^a(V)$ may be unbounded.

\begin{theorem}
	\thlabel{th:l1ho}
	Let $0<a<\infty$ and let $V$ be an a.e. non-negative potential which satisfies, for some $c>0$ the estimate $c|x|^2 \lesg V(x).$ Assume that for all $\varepsilon >0$ we have $V(x) \lesssim_{\varepsilon} e^{\varepsilon |x|^2}.$ If
	\[
		J^a(V)(x) \lesssim_g \qquad \text{and} \qquad I^a(V)(x) \lesssim_g 1,
	\]
	then $R_V^a$ is bounded on $L^1.$
\end{theorem}
\begin{proof}
We use the splitting \eqref{eq:LGsplit} again. The estimate $\norm{G}_\infty \lesssim 1$ is a consequence of \thref{lem28}. Indeed, the assumption $V(x)\lesssim e^{\varepsilon |x|^2}$ with arbitrarily small $\varepsilon>0$ implies that we can apply \thref{lem28} with arbitrarily large $a>0.$ The bound $\norm{L}_\infty \lesssim 1$ follows from the assumptions and \thref{lem:l101} as in the proof of \thref{th:l1}.
\end{proof}

	As a corollary of \thref{th:l1,th:l1ho} we obtain the $L^1(\R^d)$ boundedness of $R_V^a$ for various classes of potentials. The corollary below is a restatement of \thref{thm:exL1int} from the introduction.
	\begin{cor}
		\thlabel{cor:l1}
		Let  $V\colon \R^d\to [0,\infty)$ be a function in $L_{\rm loc}^\infty$. Then in all the three cases
		\begin{enumerate}
			\item $V(x)\approx 1$ globally
			\item For some $\alpha>0$ we have $V(x)\approx |x|^{\alpha}$ globally
			\item For some $\beta>1$ we have $V(x)\approx \beta^{|x|}$ globally
		\end{enumerate}
		each of the Riesz transforms $R_V^a,$ $a>0$, is bounded on $L^{1}(\R^d).$ 
	\end{cor}

	\begin{remark}
		Similarly to \thref{cor:linf} the Euclidean norm $\abs{\cdot}$ in (2) and (3) can be replaced by an arbitrary norm on $\R^d$.
	\end{remark}

	\begin{proof}
				In the proof implicit constants in $\lesssim,$ $\gtrsim,$ and $\approx$ do not depend on $x\in \R^d$ but may depend on $a>0,$ $\alpha>0$ or $\beta>1.$
		
		Note that in all three cases the assumptions of 	\thref{lem21} are satisfied so that the semigroup $e^{-tL}$ satisfies \eqref{eq:ed}.
		
		In case 1) we merely use \eqref{eq:ed} and obtain 
		\[
		\frac{1}{\Gamma(a)}\int_0^{\infty} e^{-tL}(V^a)(x)t^{a-1}\,dt\lesssim \frac{1}{\Gamma(a)}\int_0^{\infty} \|e^{-tL}(\ind{})\|_{\infty}\,t^{a-1}\,dt\lesssim 1,
		\]
		uniformly in $x\in \R^d.$

	In the treatment of the remaining cases we will apply \thref{th:l1} in case 2) and \thref{th:l1ho} in case 3).
		
		We start with case 2); the task is to check that the assumptions of \thref{th:l1} hold. Clearly \eqref{eq:expgro} is true for any $\eta > 0$. In the proof of \thref{cor:linf} we justified in \eqref{eq:xaIbound} that $I^b(V)(x) \lesg 1$ for any $b>0.$ Finally we need to control $K_c^a(V)(x).$ To this end we shall estimate $\sigma_x(s)$ from below. Let $C$, $N$, $m$ and $M$ be non-negative constants such that
			\[
				m\abs{x}^\alpha < V(x) < M\abs{x}^\alpha \quad \text{for a.e.} \quad \abs{x} > N
			\]
			and
			\[
				V(x) \leqslant C \quad\text{for a.e.} \quad \abs{x} \leqslant N.
			\]
			Take $|x|\geqslant N$ and assume that $|x-y|<\varepsilon \abs{x}  s^{1/\alpha},$ where $\varepsilon>0$ is a constant to be determined in a moment. Then  
			\[
				\abs{y}\leqslant \abs{x} + \abs{x-y} \leqslant \abs{x} (1+\varepsilon s^{1/\alpha})
			\]
			so that for $\abs{y} > N$ we have
			\[
				V(y) \leqslant M|y|^{\alpha} \leqslant M \abs{x}^\alpha \left( 1+\varepsilon s^{1/\alpha} \right)^\alpha \leqslant M A \abs{x}^\alpha \left( 1 + \varepsilon^\alpha s \right)
			\]
			for some constant $A \ge 1$ depending only on $\alpha.$ 
			On the other hand 
			\[
				V(x) \geqslant m\abs{x}^\alpha 
			\]
			so taking $\varepsilon$ such that $MA\varepsilon^{\alpha}=m/2$ we see that the inequality  $|x-y|<\varepsilon \abs{x}  s^{1/\alpha}$ implies 
			$$V(y)\leqslant M A \abs{x}^\alpha \left( 1 + \varepsilon^\alpha s \right)\leqslant MA|x|^{\alpha} + sV(x)/2 \leqslant \left(\frac{MA}{m} + \frac{s}{2}\right) V(x) \leqslant  sV(x),$$
			whenever $s$ is large enough (independently of $x$). Thus we proved that $\sigma_x(s)\geqslant \varepsilon \abs{x}  s^{1/\alpha}$ for such $s$ and a.e. $\abs{x} \geqslant N$.
			Consequently, 
			\[
				K_c^a(V)(x) \lesssim_g 1 +  \int_1^{\infty} e^{-c\varepsilon \abs{x}  s^{1/\alpha}}s^{a-1}\,ds \lesg 1
			\]
		for any $a,c > 0$ and an application of \thref{th:l1} completes the proof in case 2).
		
		Finally we justify case 3). It is clear that $c\abs{x}^2 \lesssim_g V(x) \lesssim e^{\varepsilon \abs{x}^2}$ for some $c > 0$ and all $\varepsilon > 0$. Moreover, in the proof of \thref{cor:linf} in \eqref{eq:bxIbound} we justified that $I^{a}(V)(x) \lesg 1$. Thus, in order to use \thref{th:l1ho} it remains to estimate $J^a(V)(x)$. Similarly, to case 2) we shall estimate $\sigma_x(s)$ from below. Let $M>0$ be a constant such that $V(y)\leqslant M\beta^{|y|},$ for a.e. $y\in \R^d$ and let  $N,$ $m$ be non-negative constants such that $m\beta^{|x|} < V(x)$ for a.e. $\abs{x} \geqslant N.$ Take $|x|\geqslant N,$ $s\geqslant 1$ and assume that $|x-y|<\frac12 \log_{\beta} s.$ Then we have 
		$
			|y|\leqslant \abs{x} +\frac12 \log_{\beta} s,
		$
		so that
		\[
			V(y)\leqslant  Ms^{1/2}\beta^{|x|}\leqslant \frac{M}{m}s^{1/2} V(x)\leqslant sV(x),
		\]
		for $s$ large enough (independently of $y$ and $x$). In other words we proved that $\sigma_x(s)\geqslant \frac12 \log_{\beta} s$ whenever $|x|\geqslant N$ and $s$ is uniformly large enough. Consequently,
		\[
			J^a(V)(x) \lesssim_g 1+\int_1^{\infty} e^{-(\log_{\beta} s)^2/32}s^{a-1}\,ds \lesg 1
		\]
		for any $a>0$ and an application of \thref{th:l1ho} completes the proof in case 3).
		
	\end{proof}
		
We finish this section with improved results for Riesz transforms $R_V^a$ in the range $0<a<1.$ These results are not needed in the proof of \thref{cor:l1}, however they might by useful in other cases.

Using the $L^1$ boundedness of $R_V^1$ one may improve \thref{lem:l101lt} in the range $0\leqslant a\leqslant 1$.

\begin{pro}
	\thlabel{lem:l11inf}
	Let $a \leqslant 1$ and assume that $e^{-tL}$ satisfies \eqref{eq:ed} with some $\delta>0$. Then the estimate
	\begin{equation} \label{lem:l11inf:eq1}
		\int_1^\infty e^{-tL}(V^{a})(x) \, t^{a-1} dt \lesssim 1
	\end{equation}
	holds uniformly in $x\in \R^d.$
\end{pro}

\begin{proof}
	Observe that for $a \leqslant 1$ we have
	\[
	e^{-tL}(V^a)(x) \leqslant e^{-tL}(V)(x) + e^{-tL}(\ind{})(x),
	\]
	so that
	\begin{equation} \label{pro:43:eq1}
		\int_1^\infty e^{-tL}(V^{a})(x) \, t^{a-1} dt \leqslant \int_1^\infty e^{-tL}(V)(x) \, t^{a-1} dt + \int_1^\infty e^{-tL}(\ind{})(x) \, t^{a-1} dt.
	\end{equation}
	From e.g.\ \cite[Theorem 4.3]{AuBA} we see that the operator $R_V^1$ is bounded on $L^1$ which, by duality, means that the first integral in \eqref{pro:43:eq1} is bounded independently of $x$. Boundedness of the second integral follows from \eqref{eq:ed}.
\end{proof}	  

Finally, combining \thref{lem:l11inf} and \thref{lem:l101} we obtain an improved version of \thref{th:l1} in the range $0<a\leqslant 1.$
	
	\begin{theorem}
		\thlabel{th:l1a<1}
		Let $0<a\leqslant 1$ and let $V$ be an a.e. non-negative potential which satisfies the growth estimate $V(x)\lesssim \exp(|x|^2/(4a))$ and such that $e^{-tL}$ has an exponential decay \eqref{eq:ed} for some $\delta>0.$  If
		\[
		J^a(V)(x) \lesssim_g 1 \qquad \text{and} \qquad I^a(V)(x) \lesssim_g 1,
		\]
		then $R_V^a$ is bounded on $L^1.$
	\end{theorem}
	\begin{proof}
		We use the splitting \eqref{eq:LGsplit}. The estimate $\norm{G}_\infty \lesssim 1$ is an immediate consequence of \thref{lem:l11inf}. The bound $\norm{L}_\infty \lesssim 1$ follows from the assumptions and \thref{lem:l101} as in the proof of \thref{th:l1}.
\end{proof}
			
\bibliographystyle{plain}
\bibliography{bib}
	
\end{document}